\documentclass[12pt]{amsart}
\usepackage[all]{xy}
\usepackage[mathscr]{euscript}
\usepackage{hyperref}
\hypersetup{colorlinks=true,urlcolor=blue,citecolor=blue,linkcolor=blue}
\usepackage{courier}
\usepackage{amsmath,amscd,amssymb,amsthm,latexsym,longtable,diagram, picture}
\usepackage{graphicx}
\usepackage{array}
\usepackage{color}
\usepackage{enumerate}
\usepackage{nicefrac}
\usepackage{listings}
\usepackage{latexsym,bm,bbm,mathrsfs}
\usepackage{hyperref,graphicx, enumerate}
\usepackage{epsfig, xcolor, graphicx}
\usepackage[shortlabels]{enumitem}
\usepackage{indentfirst, setspace}
\usepackage{tikz-cd}

\allowdisplaybreaks
\usepackage{caption}
\usepackage{cancel}
\usepackage[normalem]{ulem}

\usetikzlibrary{calc,matrix,arrows,decorations.markings}

\newcommand{\bbz}{\mathbb{Z}}
\newcommand{\bbq}{\mathbb{Q}}

\newcommand{\bbr}{\mathbb{R}}

\newcommand{\mat}{\begin{pmatrix}}
\newcommand{\emat}{\end{pmatrix}}

\newtheorem{theorem}{Theorem}[section]

\newtheorem{corollary}[theorem]{Corollary}
\newtheorem{lemma}[theorem]{Lemma}
\newtheorem{rmk}[theorem]{Remark}

\newtheorem{conjecture}[theorem]{Conjecture}

\begin{document}
 \title[Additive Rigidity for $x$-Coordinates of Rational Points on Elliptic Curves]
{Additive Rigidity for $x$-Coordinates of Rational Points on Elliptic Curves}

\author{Seokhyun Choi}

\address{
Dept. of Mathematical Sciences, KAIST,
291 Daehak-ro, Yuseong-gu,
Daejeon 34141, South Korea
}
\email{sh021217@kaist.ac.kr}

\date{\today}

\begin{abstract}
    We study the interaction between the group law on an elliptic curve and the additive structure of $x$-coordinates of rational points on an elliptic curve. Let $E/\mathbb{Q}$ be an elliptic curve of Mordell-Weil rank $r \geq 1$, $d \geq 1$ be an integer, and $0<\rho \leq 1$. We show that if a $d$-dimensional proper generalized arithmetic progression in $\mathbb{Q}$ contains the $x$-coordinates of rational points on $E/\bbq$ with positive proportion $\rho$, then the number of such points is bounded by $A(E,d,\rho)^r$. The proof combines extraction lemmas, gap principles, and the bounds for spherical codes. As an application, we obtain restrictions on sets of rational points whose $x$-coordinates have small sumsets or large additive energy.
\end{abstract}

\maketitle

\section{Introduction}\label{Introduction}

Many problems in number theory arise from the interaction of distinct algebraic structures. 
Although such problems are often easy to formulate, they frequently lead to deep and difficult questions. A classical example occurs in the study of rational points on elliptic curves, where two different additive structures naturally appear: the group law on the Mordell-Weil group $E(\bbq)$ and the additive structure of the rational numbers $\bbq$ itself.

These two structures behave in very different ways. 
While Mordell's theorem asserts that $E(\bbq)$ is a finitely generated abelian group, the additive group $\bbq$ admits highly structured subsets such as long arithmetic progressions and generalized arithmetic progressions. 
Understanding how these two additive structures interact naturally leads to questions about additive patterns among the coordinates of rational points on elliptic curves.

One of the simplest instances of such a problem concerns arithmetic progressions among the $x$-coordinates of rational points. 
Given a finite sequence of rational points on an elliptic curve, one may ask whether their $x$-coordinates can form a long arithmetic progression.

To make this notion precise, let $E/\bbq$ be an elliptic curve and choose a Weierstrass equation
\[y^2 + a_1 xy + a_3 y = x^3+a_2x^2+a_4x+a_6,
\qquad a_i\in\bbz.\]
Given a finite sequence of rational points $\{P_1,\ldots,P_N\}\subseteq E(\bbq)$, we say that $\{P_1,\ldots,P_N\}$ is in \emph{$x$-arithmetic progression} if the set of $x$-coordinates
\[\{x(P_1),\ldots,x(P_N)\}\]
forms an arithmetic progression in $\bbq$.

This notion is independent of the choice of Weierstrass equation. 
Indeed, if one chooses another Weierstrass equation
\[y^2 + a_1'xy + a_3'y = x^3+a_2'x^2+a_4'x+a_6',\]
then the change of variables is given by
\[x\mapsto u^2x+r,\qquad y\mapsto u^3y+u^2sx+t\]
for some $u\in\bbq^\ast$ and $r,s,t\in\bbq$. 
Since affine transformations preserve arithmetic progressions, the notion of an $x$-arithmetic progression does not depend on the choice of Weierstrass equation.

Bremner (\cite{Bre99}) formulated the following conjecture concerning such progressions.

\begin{conjecture}[Bremner]\label{Bremner}
There exists an absolute constant $A>0$ such that for every elliptic curve $E/\bbq$ of rank $r$ and for every sequence $\{P_1,\ldots,P_N\}$ of rational points in $x$-arithmetic progression,
\[N \leq A^r .\]
\end{conjecture}

Some evidence supporting this conjecture is known. See, for instance, \cite{Bre99}, \cite{Cam03}, \cite{GT05}, \cite{MZ17}, and \cite{Spe11}. In particular, Bremner, Silverman, and Tzanakis \cite{BST00} proved that when one restricts to the quadratic twist family of a fixed elliptic curve and considers integral points in $x$-arithmetic progression lying in a rank $1$ subgroup, then the length of such a progression is uniformly bounded. 

More recently, Garc\'ia-Fritz and Past\'en \cite{GP21} obtained further progress toward Conjecture~\ref{Bremner}, establishing the conjecture for families of elliptic curves with fixed $j$-invariant. Their argument used R\'{e}mond's theorem, whose constants depend on the height of the ambient abelian variety. In a recent note \cite{GP26}, they explain that replacing R\'{e}mond's theorem by the uniform Mordell-Lang conjecture of Gao-Ge-K\"uhne \cite{GGK21} removes this dependence and yields Bremner's conjecture~\ref{Bremner}. They also give a short proof of the corresponding uniformity consequence, namely that uniformly bounded ranks imply uniformly bounded lengths of $x$-arithmetic progressions, using the theorem of Dimitrov-Gao-Habegger \cite{DGH21}.

The purpose of this paper is to study a broader phenomenon underlying Conjecture~\ref{Bremner}. 
Rather than restricting attention to arithmetic progressions, we investigate the interaction between rational points on elliptic curves and sets possessing additive structure in the sense of additive combinatorics.

From the viewpoint of additive combinatorics, generalized arithmetic progressions provide a flexible model for highly structured subsets of abelian groups. 
Many inverse theorems show that sets exhibiting strong additive properties are efficiently described by generalized arithmetic progressions of bounded dimension. 
It is therefore natural to ask how rational points on elliptic curves may populate such sets.

A $d$-dimensional generalized arithmetic progression in $\bbq$ is a set of the form
\[\{a_0 + k_1a_1 + \cdots + k_da_d\:|\:0 \leq k_1 < N_1,\ldots,0 \leq k_d < N_d\}\]
where $a_0,\ldots,a_d\in\bbq$ and $N_1,\ldots,N_d$ are positive integers. A generalized arithmetic progression is called proper if the cardinality of the set 
\[\{a_0 + k_1a_1 + \cdots + k_da_d\:|\:0 \leq k_1 < N_1,\ldots,0 \leq k_d < N_d\}\]
equals $N_1 \cdots N_d$.

Our main theorem shows that if a proper generalized arithmetic progression contains a positive proportion of the $x$-coordinates of rational points on an elliptic curve, then the total number of such points is strongly restricted. This shows that the additive structure of $\bbq$ and the group structure of $E(\bbq)$ are fundamentally incompatible.

\begin{theorem}\label{main_theorem}
Let $E/\bbq$ be an elliptic curve of Mordell-Weil rank $r \geq 1$. Let $d \geq 1$ be an integer and $0<\rho \leq 1$. Then there exists an effectively computable constant $A(E,d,\rho)>0$ with the following property.

For any finite subset $\mathcal P \subseteq E(\bbq)$ such that
\begin{enumerate}
    \item the set of $x$-coordinates $x(\mathcal P)$ is contained in a $d$-dimensional proper generalized arithmetic progression $G$, and
    \item $\lvert x(\mathcal P)\rvert \geq \rho \lvert G \rvert$,
\end{enumerate}
we have
\[\lvert \mathcal P \rvert \leq A(E,d,\rho)^{r}.\]

Moreover, assume that Conjecture~\ref{Lang_conjecture} holds with an explicitly given admissible constant $c_L>0$. Then $A(E,d,\rho)$ may be taken to be an effectively computable constant depending only on $d$ and $\rho$.
\end{theorem}

The case $d=1$ and $\rho=1$ recovers the type of additive pattern considered in Bremner's conjecture~\ref{Bremner}. The novelty of Theorem~\ref{main_theorem} is that it treats positive-density containment in proper generalized arithmetic progressions, with constants that are effective for a fixed curve and conditionally uniform and effective under Lang's conjecture~\ref{Lang_conjecture}.

\begin{rmk}
    The proof of Theorem~\ref{main_theorem} gives a slightly more general statement. Let $\Gamma \subseteq E(\bbq)$ be a subgroup of rank $s$. Then the same conclusion holds for finite subsets $\mathcal P \subseteq \Gamma$, with the exponent $r$ replaced by $s$. Indeed, the argument only uses the Mordell-Weil lattice generated by the points under consideration. The spherical-code bounds are applied in $\Gamma \otimes_\bbz \bbr$, whose dimension is $s$, and the other arguments are unchanged.
\end{rmk}

As immediate consequences, combining Theorem~\ref{main_theorem} with Freiman's theorem or Balog-Szemer\'{e}di-Gowers theorem from additive combinatorics yields strong restrictions on sets of rational points whose $x$-coordinates have small sumsets or large additive energy.

\begin{corollary}\label{Freiman_cor_intro}
    Let $E/\bbq$ be an elliptic curve of Mordell-Weil rank $r \geq 1$ and let $\mathcal P \subseteq E(\bbq)$ be a finite subset. Put $S=x(\mathcal P)$. 
    
    Suppose that
    \[\lvert S+S \rvert \leq K\lvert S \rvert\]
    for some constant $K>0$. Then there exists a constant $A(E,K)>0$ such that
    \[\lvert \mathcal P \rvert \leq A(E,K)^r.\]

    Moreover, assume that Conjecture~\ref{Lang_conjecture} holds with an explicitly given admissible constant $c_L>0$. Then $A(E,K)$ may be taken to be an effectively computable constant depending only on $K$.
\end{corollary}

\begin{corollary}
    Let $E/\bbq$ be an elliptic curve of Mordell-Weil rank $r \geq 1$ and let $\mathcal P \subseteq E(\bbq)$ be a finite subset. Put $S=x(\mathcal P)$. 
    
    Suppose that 
    \[\lvert E(S) \rvert \geq \frac{\lvert S \rvert^3}{K}\]
    for some constant $K>0$. Then there exists a constant $A(E,K)>0$ such that
    \[\lvert \mathcal P \rvert \leq A(E,K)^r.\]

    Moreover, assume that Conjecture~\ref{Lang_conjecture} holds with an explicitly given admissible constant $c_L>0$. Then $A(E,K)$ may be taken to be an effectively computable constant depending only on $K$.
\end{corollary}

Our argument relies on gap principles for rational points of large canonical height, originating in work of Helfgott and Venkatesh (\cite{HV06}). These principles imply that rational points with comparable canonical heights must be separated by a definite angle in the Mordell-Weil lattice. Combined with bounds for spherical codes, this leads to quantitative limits on the number of such points.

Finally, we recall Lang's conjecture on canonical heights (\cite{Lan78}, \cite{Sil84}), which appears in the statement of the theorem. It predicts that the canonical height of every non-torsion rational point is bounded below by a constant multiple of the height of the $j$-invariant or the minimal discriminant of the curve. This conjecture plays a fundamental role in the study of the distribution of rational points on elliptic curves.

\begin{conjecture}\label{Lang_conjecture}
    There exists an absolute constant $c_L$ such that for every elliptic curve $E/\bbq$ with $j$-invariant $j_E$ and minimal discriminant $\Delta_E$,
    \[\hat h(P) \geq c_L\max\{h(j_E),h(\Delta_E)\}\]
    for all non-torsion points $P\in E(\bbq)$.
\end{conjecture}

Consequently, for any family $\mathscr F$ of elliptic curves satisfying one of the following conditions, the dependence on $E$ in Theorem~\ref{main_theorem} can be removed; that is, the constant $A(E,d,\rho)$ may be chosen to be an effectively computable constant depending only on $d$ and $\rho$:
\begin{enumerate}[(i)]
    \item elliptic curves with integral $j$-invariant;
    \item elliptic curves with bounded Szpiro ratio; or
    \item quadratic twist families of elliptic curves.
\end{enumerate}
In these cases Lang's conjecture is known to hold uniformly.

The paper is organized as follows. Section~\ref{Mordell-Weil_Geometry} reviews canonical heights and the geometry of the Mordell-Weil lattice. Section~\ref{Gap_principles} develops the gap principles used in the argument. Section~\ref{Extraction_lemma} proves an extraction lemma concerning generalized arithmetic progressions. Section~\ref{Reduction} reduces Theorem~\ref{main_theorem} to two theorems according to small $x$-coordinates and large $x$-coordinates. Section~\ref{Proof1} and Section~\ref{Proof2} prove the corresponding theorems, respectively. Finally, Section~\ref{Applications} provides several applications of Theorem~\ref{main_theorem}.

\section{The Geometry of the Mordell-Weil Lattice}\label{Mordell-Weil_Geometry}

\subsection{Mordell-Weil geometry}

Let $E/\bbq$ be an elliptic curve of Mordell-Weil rank $r$. Then $E(\bbq) \otimes_\bbz \bbr$ may be naturally identified with $\bbr^r$ and the canonical height $\hat{h}$ on $E(\bbq)$ extends $\bbr$-linearly to a positive definite quadratic form on this space. Thus the canonical height $\hat{h}$ endows the real vector space 
\[E(\bbq) \otimes_\bbz \bbr \cong \bbr^r\]
with the structure of a Euclidean space of dimension $r$, in which rational points may be viewed as vectors.

Let $P,Q \in E(\bbq)$ be rational points. The inner product $\langle P,Q \rangle$ is defined by 
\[\langle P,Q \rangle := \frac{1}{2}(\hat{h}(P+Q) - \hat{h}(P) - \hat{h}(Q)),\]
and the norm $\lVert P \rVert$ is defined by 
\[\lVert P \rVert := \sqrt{\langle P,P \rangle} = \sqrt{\hat{h}(P)}.\]
If $P,Q$ are non-torsion, the angle $\theta_{P,Q}$ between $P,Q$ is defined by the formula 
\[\cos \theta_{P,Q} := \frac{\langle P,Q \rangle}{\lVert P \rVert\lVert Q \rVert} = \frac{\hat{h}(P+Q) - \hat{h}(P) - \hat{h}(Q)}{2\sqrt{\hat{h}(P)\hat{h}(Q)}}.\]

Suppose a finite set $X$ of non-torsion points in $E(\bbq)$ satisfies 
\[\cos \theta_{P,Q} \leq \cos \theta_0,\quad P,Q \in X\]
for some $\theta_0>0$. Then the image of $X$ under 
\[X \longrightarrow E(\bbq) \otimes_\bbz \bbr,\quad P \longmapsto P \otimes \frac{1}{\sqrt{\hat{h}(P)}}\]
is a finite set of unit vectors with uniform angular separation. After normalization by height, collections of rational points with uniform angular separation may therefore be regarded as spherical codes, which are introduced in the next subsection.

\subsection{Spherical codes}

Let $\Omega_r$ denote the unit sphere in $\bbr^r$. A finite subset $X \subseteq \Omega_r$ is called a spherical code with minimal angle $\theta$ if $\langle x,y \rangle \leq \cos \theta$ for every distinct $x,y \in X$. We write $A(r,\theta)$ for the maximum size of the spherical code $X$. We recall two standard bounds for $A(r,\theta)$; one for $0<\theta<\pi/2$ and one for $\theta>\pi/2$.

\begin{theorem}\label{spherical_code_bound1}
    For fixed $0<\theta<\pi/2$, 
    \[\frac{1}{r}\log A(r,\theta) \leq \frac{1+\sin \theta}{2\sin \theta}\log\frac{1+\sin \theta}{2\sin \theta} - \frac{1-\sin \theta}{2\sin \theta}\log \frac{1-\sin \theta}{2\sin \theta} + o(1),\]
    where $o(1) \rightarrow 0$ as $r \rightarrow \infty$ and $o(1)$ is explicit for $\theta$.

    Consequently, there exists an effectively computable constant $c(\theta)$ such that 
    \[A(r,\theta) \leq c(\theta)^r.\]
\end{theorem}
\begin{proof}
    See \cite{KL78}.
\end{proof}

\begin{theorem}\label{spherical_code_bound2}
    For fixed $\theta>\pi/2$, there exists an effectively computable constant $c(\theta)$ such that 
    \[A(r,\theta) \leq c(\theta).\]
\end{theorem}
\begin{proof}
    See \cite[Chapter~1]{CS99}.
\end{proof}

The above discussion shows that any arithmetic mechanism producing angular separation between rational points imposes packing constraints in the Mordell-Weil lattice. In the next section we develop such a mechanism, which we call the gap principle.

\section{Gap principles}\label{Gap_principles}

The geometric framework developed in the previous section shows that large collections of rational points can be controlled once uniform angular separation is available in the Mordell-Weil lattice. The purpose of this section is to establish such separation results via gap principles for rational points on elliptic curves.

Roughly speaking, gap principles assert that rational points of comparable canonical height cannot lie too close to each other in the Mordell-Weil lattice unless strong arithmetic degeneracies occur. The formulations obtained here are adapted to the additive structures that will arise later in the paper.

The gap principles used in this paper originate in work of Helfgott \cite{Hel04} and subsequent refinements such as \cite{Alp14}. Earlier formulations were primarily suited to integral points. The versions established here are adapted to rational points and to the additive configurations considered in the present work.

\subsection{Weierstrass models and global height parameters}

In order to obtain explicit Diophantine estimates, we fix throughout the paper a Weierstrass
model for the elliptic curve and introduce a global height parameter governing its arithmetic complexity.

Let $E/\bbq$ be an elliptic curve and set 
\[M_E = \max\{h(j_E),h(\Delta_E)\}\]
where $j_E$ denotes the $j$-invariant of $E$ and $\Delta_E$ its minimal discriminant. 

We choose a minimal Weierstrass equation
\[E:y^2 + a_1 xy + a_3 y = x^3+a_2x^2+a_4x+a_6,\quad a_i \in \bbz\]
for $E$. After a standard change of variables 
\[x \longmapsto \frac{1}{36}(x-3b_2),\quad y \longmapsto \frac{1}{2}\left(\frac{y}{108}-\frac{a_1}{36}(x-3b_2)-a_3\right),\] 
where $b_2=a_1^2+4a_2$, this equation may be written in short Weierstrass form 
\[E:y^2 = x^3+Ax+B,\quad A,B \in \bbz.\]
The discriminant is changed by $\Delta = 6^{12}\Delta_E$. 

We define the global parameter 
\[X = \max\{\lvert A \rvert^3,\lvert B \rvert^2\}\] 
which will serve as a uniform scale for all height estimates appearing in the paper. The discriminant and $j$-invariant satisfy 
\[\Delta = -16(4A^3+27B^2),\quad j = 1728\frac{4A^3}{4A^3+27B^2}.\]
Consequently, standard height estimates yield 
\begin{equation}\label{Delta_j_upper_bound}
    h(\Delta) \leq \log X + 6.21,\quad h(j) \leq \log X + 8.85
\end{equation}
and 
\begin{equation}\label{Delta_j_lower_bound}
    \log X \leq h(\Delta)+h(j) + 0.7.
\end{equation}
Note that $h(\Delta_E) \geq \log 11$. Therefore, we have 
\begin{equation}\label{M_E_lower_bound}
    M_E \geq 2.39.
\end{equation}
and 
\begin{equation}\label{X_canonical_lower_bound}
    \log X \geq 17.68.
\end{equation}
Also \eqref{Delta_j_upper_bound} and \eqref{Delta_j_lower_bound} imply 
\begin{equation}\label{X_lower_bound}
    M_E \leq \log X + 8.85 \leq 2\log X
\end{equation}
and 
\begin{equation}\label{X_upper_bound}
   \log X \leq 2M_E + 43.71 \leq 21M_E.
\end{equation}
In particular, the quantities $M_E$ and $\log X$ are comparable up to absolute constants.

Throughout the paper we therefore work with the fixed integral model
\[E:y^2=x^3+Ax+B,\quad A,B\in\bbz,\]
and express all height estimates in terms of the global parameter $X$ introduced above. This normalization allows the gap principles established below to be formulated uniformly.

\subsection{Height estimates}

In this subsection we establish explicit height estimates that will later translate into angular separation in the Mordell-Weil lattice.

For a rational point $P \in E(\mathbb Q)$, we define the Weil height $h(P)$ by 
\[h(P) := h(x(P))\]
and we define the canonical height $\hat{h}(P)$ by 
\[\hat{h}(P) := \lim_{n \rightarrow \infty} \frac{h(2^nP)}{4^n}.\]
We note that this normalization differs from the convention including an additional factor $1/2$.

We begin with a standard comparison between the Weil height $h$ and the canonical height $\hat{h}$, which allows us to pass between intrinsic and explicit height estimates.

\begin{lemma}\label{Weil_canonical_height}
    Let $P \in E(\bbq)$. Then 
    \[-\frac{5}{12}\log X -5.2 \leq \hat{h}(P) - h(P) \leq \frac{1}{3}\log X+4.65.\]
    In particular, 
    \[-\frac{3}{4}\log X \leq \hat{h}(P) - h(P) \leq \frac{2}{3}\log X.\]
\end{lemma}
\begin{proof}
    By \cite{Sil90},  
    \[-\frac{1}{4}h(j)-1.946 - \frac{1}{6}h(\Delta) \leq \hat{h}(P) - h(P) \leq \frac{1}{6}h(j) + 2.14 + \frac{1}{6}h(\Delta).\]
    By \eqref{Delta_j_upper_bound}, the first statement follows. For the second statement, apply \eqref{X_canonical_lower_bound}.
\end{proof}

The next lemma provides the basic Diophantine height inequality underlying the gap principles proved later.

\begin{lemma}\label{Weil_height_estimate}
    Let $0 \leq \delta \leq 1$ be a fixed constant. Let $P,Q \in E(\bbq)$ satisfy $X^{1/6} \leq x(P)<x(Q)$ and 
    \[x(P) = \frac{x_1}{s},\quad x(Q) = \frac{x_2}{s}\]
    where $x_1,x_2,s \in \bbz$ satisfy $\gcd(x_1,s) \leq s^{\delta}$, $\gcd(x_2,s) \leq s^{\delta}$.
    Then 
    \begin{equation}\label{weil0}
        h(P+Q) \leq h(P)+2h(Q)+3\delta h(s)+2.9.
    \end{equation}
\end{lemma}
\begin{proof}
    We have 
    \begin{align*}
        x(P+Q) &= \left(\frac{y(P)-y(Q)}{x(P)-x(Q)}\right)^2 - (x(P)+x(Q)) \\
        &= \frac{(x(P)x(Q)+A)(x(P)+x(Q))+2B - 2y(P)y(Q)}{(x(P)-x(Q))^2} \\
        &= \frac{(x_1x_2+s^2A)(x_1+x_2)+2s^3B - 2s^3y(P)y(Q)}{s(x_1-x_2)^2}.
    \end{align*}
    By using the inequalities 
    \[\lvert A \rvert \leq X^{1/3},\quad \lvert B \rvert \leq X^{1/2}.\]
    and estimates 
    \[h(x+y) \leq \max\{h(x),h(y)\} + \log 2,\quad h(xy) \leq h(x)+h(y),\]
    we have 
    \begin{gather*}
        h((x_1x_2+s^2A)(x_1+x_2)) \leq h(x_1)+2h(x_2)+2\log 2, \\
        h(2s^3B) \leq h(x_1)+2h(x_2)+\log 2, \\
        h(2s^3y(P)y(Q)) \leq h(x_1)+2h(x_2)+\log 6.
    \end{gather*}
    For the estimate involving the term $y(P)y(Q)$, we additionally used the relations 
    \[s^3y(P)^2 = x_1^3+s^2Ax_1+s^3B,\quad s^3y(Q)^2 = x_2^3+s^2Ax_2+s^3B,\]
    which follow from the defining equation of $E$. Therefore, 
    \[h((x_1x_2+s^2A)(x_1+x_2)+2s^3B - 2s^3y(P)y(Q)) \leq h(x_1)+2h(x_2)+\log 18.\]
    Since $x_1<x_2$, 
    \[h(s(x_1-x_2)^2) \leq h(sx_2^2) \leq h(x_1)+2h(x_2).\]
    Hence, 
    \begin{equation}\label{weil1}
        h(x(P+Q)) \leq h(x_1)+2h(x_2)+\log 18.
    \end{equation}

    Finally, $(x_1,s) \leq s^{\delta}$ and $(x_2,s) \leq s^{\delta}$ imply 
    \begin{equation}\label{weil2}
        h(P) \geq h(x_1) - \delta h(s),\quad h(Q) \geq h(x_2) - \delta h(s).
    \end{equation}
    Combining \eqref{weil1} and \eqref{weil2} implies \eqref{weil0}.
\end{proof}

In particular, the cases $\delta=0$ and $\delta=1$ yield respectively 
\begin{equation}\label{delta=0}
    h(P+Q) \leq h(P)+2h(Q)+2.9
\end{equation}
and 
\begin{equation}\label{delta=1}
    h(P+Q) \leq h(P)+2h(Q)+3h(s)+2.9.
\end{equation}

For points with small $x$-coordinates, a different type of height estimate is required.

\begin{lemma}\label{Weil_height_estimate_small}
    Let $P,Q \in E(\bbq)$ satisfy $\lvert x(P) \rvert,\lvert x(Q) \rvert \leq 2X^{1/6}$, and 
    \[x(P) = \frac{x_1}{s},\quad x(Q) = \frac{x_2}{s}\]
    where $x_1,x_2,s \in \bbz$ satisfy $x_1 \neq x_2$. Then 
    \[h(P+Q) \leq 3h(s) + \frac{1}{2}\log X+3.9.\]
\end{lemma}
\begin{proof}
    As in Lemma~\ref{Weil_height_estimate} we have 
    \begin{align*}
        x(P+Q) = \frac{(x_1x_2+s^2A)(x_1+x_2)+2s^3B - 2s^3y(P)y(Q)}{s(x_1-x_2)^2}.
    \end{align*}
    Similar estimates as in Lemma~\ref{Weil_height_estimate} give 
    \begin{gather*}
        h((x_1x_2+s^2A)(x_1+x_2)) \leq 3h(s) + \frac{1}{2} \log X + 5\log 2, \\
        h(2s^3B) \leq 3h(s) + \frac{1}{2} \log X +\log 2, \\
        h(2s^3y(P)y(Q)) \leq 3h(s) + \frac{1}{2} \log X +4\log 2 + \log 3.
    \end{gather*}
    Therefore, 
    \[h((x_1x_2+s^2A)(x_1+x_2)+2s^3B - 2s^3y(P)y(Q)) \leq 3h(s) + \frac{1}{2} \log X +\log 48.\]
    Since $\lvert x_1-x_2 \rvert \leq 4sX^{1/6}$, 
    \[h(s(x_1-x_2)^2) \leq 3h(s) + \frac{1}{3}\log X + 4\log 2.\]
    Hence, 
    \[h(x(P+Q)) \leq 3h(s) + \frac{1}{2}\log X +\log 48.\]
\end{proof}

The preceding estimates control the growth of heights under the group law in terms of the arithmetic data of the $x$-coordinates. In the next subsection, these inequalities will be converted into quantitative bounds for angles between rational points in the Mordell-Weil lattice, leading to the gap principles central to this work.

\subsection{Gap principles for large $x$-coordinates}

We now convert the height inequality of Lemma~\ref{Weil_height_estimate} into angular separation in the Mordell-Weil lattice. The argument naturally splits according to whether the common denominator $s$ is small or large relative to the global scale $\log X$, leading to the following two gap principles.

\begin{theorem}\label{gap_principle1_small_s}
    Let $0 \leq \delta \leq 1$, $\gamma>0$, $M>0$, and $\alpha > 1$ be fixed constants. 
    Let $P,Q \in E(\bbq)$ satisfy $X^{1/6} \leq x(P)<x(Q)$ and write 
    \[x(P) = \frac{x_1}{s},\quad x(Q) = \frac{x_2}{s}\]
    where $x_1,x_2,s \in \bbz$ satisfy $\gcd(x_1,s) \leq s^{\delta}$, $\gcd(x_2,s) \leq s^{\delta}$. Assume moreover that 
    \[h(s) \leq \frac{1}{\gamma}\log X,\quad \hat{h}(P),\hat{h}(Q) > M\log X,\quad \max \left\{\frac{\hat{h}(Q)}{\hat{h}(P)},\frac{\hat{h}(P)}{\hat{h}(Q)}\right\} \leq \alpha.\] 
    Then 
    \[\cos \theta_{P,Q} \leq \frac{\sqrt{\alpha}}{2} + \frac{3\delta}{2M \gamma} + \frac{4}{M}.\]
\end{theorem}
\begin{proof}
    By Lemma~\ref{Weil_height_estimate}, 
    \[h(P+Q) \leq h(P)+2h(Q)+3\delta h(s)+2.9.\] 
    By Lemma~\ref{Weil_canonical_height}, 
    \[\hat{h}(P+Q) \leq \hat{h}(P) + 2\hat{h}(Q) + 3\delta h(s) + 4\log X.\]
    Substituting these bounds into the definition of $\cos\theta_{P,Q}$ yields 
    \[\cos \theta_{P,Q}= \frac{\hat{h}(P+Q) - \hat{h}(P) - \hat{h}(Q)}{2\sqrt{\hat{h}(P)\hat{h}(Q)}}  \leq \frac{\sqrt{\alpha}}{2} + \frac{3\delta}{2M \gamma} + \frac{4}{M}.\]
\end{proof}

\begin{theorem}\label{gap_principle1_large_s}
    Let $0 \leq \delta \leq 1$, $\gamma>0$, $M>0$, and $\alpha > 1$ be fixed constants satisfying $\delta+\gamma<1$. 
    Let $P,Q \in E(\bbq)$ satisfy $X^{1/6} \leq x(P)<x(Q)$ and write 
    \[x(P) = \frac{x_1}{s},\quad x(Q) = \frac{x_2}{s},\]
    where $x_1,x_2,s \in \bbz$ satisfy $\gcd(x_1,s) \leq s^{\delta}$, $\gcd(x_2,s) \leq s^{\delta}$. Assume moreover that 
    \[h(s) > \frac{1}{\gamma}\log X,\quad  \hat{h}(P),\hat{h}(Q) > M\log X,\quad \max \left\{\frac{\hat{h}(Q)}{\hat{h}(P)},\frac{\hat{h}(P)}{\hat{h}(Q)}\right\} \leq \alpha.\] 
    Then 
    \[\cos \theta_{P,Q} \leq \frac{\sqrt{\alpha}}{2} + \frac{3\delta}{2(1-\delta-\gamma)} + \frac{4}{M}.\]
\end{theorem}
\begin{proof}
    By Lemma~\ref{Weil_height_estimate}, 
    \[h(P+Q) \leq h(P) + 2h(Q) + 3\delta h(s) + 2.9.\]
    By Lemma~\ref{Weil_canonical_height}, 
    \[\hat{h}(P+Q) \leq \hat{h}(P) + 2\hat{h}(Q) + 3\delta h(s) + 4\log X.\]
    From $\gcd(x_1,s) \leq s^{\delta}$ and $\gcd(x_2,s) \leq s^{\delta}$, 
    \[h(P),h(Q) \geq (1-\delta)h(s).\]
    By Lemma~\ref{Weil_canonical_height}, 
    \[\hat{h}(P),\hat{h}(Q) \geq (1-\delta)h(s) - \log X \geq (1-\delta-\gamma)h(s).\]
    Substituting these bounds into the definition of $\cos\theta_{P,Q}$ yields 
    \[\cos \theta_{P,Q}= \frac{\hat{h}(P+Q) - \hat{h}(P) - \hat{h}(Q)}{2\sqrt{\hat{h}(P)\hat{h}(Q)}}  \leq \frac{\sqrt{\alpha}}{2} + \frac{3\delta}{2(1-\delta-\gamma)} + \frac{4}{M}.\]
\end{proof}

The preceding results show that rational points with large $x$-coordinates and comparable canonical
heights must be separated by a definite angle in the Mordell-Weil lattice. In the next subsection we establish analogous separation results for points with small $x$-coordinates.

\subsection{Gap principles for small $x$}

We now treat the complementary regime in which the $x$-coordinates remain small. Applying Lemma~\ref{Weil_height_estimate_small} in the case of large denominator $s$ again yields angular separation in the Mordell-Weil lattice.

\begin{theorem}\label{gap_principle2_large_s}
    Let $0 \leq \delta \leq 1$, $\gamma>0$, and $M>0$ be fixed constants satisfying $\delta+\gamma<1$. 
    Let $P,Q \in E(\bbq)$ satisfy $\lvert x(P) \rvert,\lvert x(Q) \rvert \leq 2 X^{1/6}$, and write 
    \[x(P) = \frac{x_1}{s},\quad x(Q) = \frac{x_2}{s}\]
    where $x_1,x_2,s \in \bbz$ satisfy $x_1 \neq x_2$, $\gcd(x_1,s) \leq s^{\delta}$, $\gcd(x_2,s) \leq s^{\delta}$. Assume moreover that 
    \[h(s) > \frac{1}{\gamma}\log X,\quad  \hat{h}(P),\hat{h}(Q) > M\log X.\] 
    Then 
    \[\cos \theta_{P,Q} \leq \frac{1+2\delta}{2(1-\delta-\gamma)} + \frac{2}{M}.\]
\end{theorem}
\begin{proof}
    By Lemma~\ref{Weil_height_estimate_small}, we obtain 
    \[h(P+Q) \leq 3h(s) + \frac{1}{2}\log X+3.9.\]
    From $\gcd(x_1,s) \leq s^{\delta}$ and $\gcd(x_2,s) \leq s^{\delta}$, 
    \[h(P),h(Q) \geq (1-\delta)h(s).\]
    Thus 
    \[h(P+Q) - h(P) - h(Q) \leq (1+2\delta)h(s) + \frac{1}{2}\log X+3.9.\]
    By Lemma~\ref{Weil_canonical_height}, 
    \[\hat{h}(P+Q) - \hat{h}(P) - \hat{h}(Q) \leq (1+2\delta)h(s) + 4\log X\]
    and 
    \[\hat{h}(P),\hat{h}(Q) \geq (1-\delta)h(s) - \log X \geq (1-\delta-\gamma)h(s).\]
    Substituting these bounds into the definition of $\cos\theta_{P,Q}$ yields 
    \[\cos \theta_{P,Q} = \frac{\hat{h}(P+Q) - \hat{h}(P) - \hat{h}(Q)}{2\sqrt{\hat{h}(P)\hat{h}(Q)}} \leq \frac{1+2\delta}{2(1-\delta-\gamma)} + \frac{2}{M}.\]
\end{proof}

\section{Extraction lemmas}\label{Extraction_lemma}

In this section we establish an extraction principle for dense subsets of generalized arithmetic progressions. Roughly speaking, we show that if a subset occupies a positive proportion of a generalized arithmetic progression, then a positive proportion of its elements satisfy the required primitiveness condition. This density-preserving extraction mechanism will serve as the bridge between additive structure and the gap principles developed in Section~3.

The argument proceeds in several steps. We first establish a general divisibility principle controlling products of integers whose pairwise greatest common divisors are suitably restricted. This will then be applied to arithmetic progressions of rational numbers, and finally extended inductively to generalized arithmetic progressions.

\begin{lemma}\label{Sec5_lem1}
    Let $N$ be a positive integer and $g_1,\ldots,g_n$ be integers satisfying 
    \begin{equation}\label{lem1_eq1}
        \gcd(g_i,g_j) \mid (j-i)N,\quad 1 \leq i<j \leq n.
    \end{equation}
    Then 
    \begin{equation}\label{lem1_eq2}
        g_1 \cdots g_n \mid N^{n-1}\prod_{k=1}^{n-1}k! \cdot \mathrm{lcm}(g_1,\ldots,g_n).
    \end{equation}
\end{lemma}
\begin{proof}
    We recall two standard identities 
    \begin{equation}\label{lem1_eq3}
        ab = \gcd(a,b) \cdot \mathrm{lcm}(a,b)
    \end{equation}
    and 
    \begin{equation}\label{lem1_eq4}
        \gcd(\mathrm{lcm}(a_1,\ldots,a_k),b) \mid \mathrm{lcm}(\gcd(a_1,b),\ldots,\gcd(a_k,b)).
    \end{equation}
    For the proof of \eqref{lem1_eq4}, set $A = \mathrm{lcm}(a_1,\ldots,a_k)$, $d = \gcd(A,b)$, $\gcd(a_i,b) = g_i$ for $1 \leq i \leq k$, and $L = \mathrm{lcm}(g_1,\ldots,g_k)$. Assume $d \nmid L$. Then there exists a prime $p$ such that $p^t \mid d$ but $p^t \nmid L$. As $p^t \mid A$, $p^t \mid a_i$ for some $i$. However, $p^t \mid b$, so that $p^t \mid g_i$, which implies $p^t \mid L$.

    To prove \eqref{lem1_eq2}, we will use induction on $n$. Suppose 
    \[g_1 \cdots g_n \mid N^{n-1}\prod_{k=1}^{n-1}k! \cdot \mathrm{lcm}(g_1,\ldots,g_n).\]
    We have to prove 
    \[g_1 \cdots g_ng_{n+1} \mid N^n\prod_{k=1}^nk! \cdot \mathrm{lcm}(g_1,\ldots,g_n,g_{n+1}).\]
    By \eqref{lem1_eq3}, it suffices to prove 
    \[\gcd(\mathrm{lcm}(g_1,\ldots,g_n),g_{n+1}) \mid Nn!.\]
    By \eqref{lem1_eq4} and \eqref{lem1_eq1}, we have 
    \[\gcd(\mathrm{lcm}(g_1,\ldots,g_n),g_{n+1}) \mid \mathrm{lcm}(N,\ldots,nN) \mid Nn!,\]
    which ends the proof.
\end{proof}

Lemma~\ref{Sec5_lem1} provides a multiplicative constraint showing that restrictions on pairwise greatest common divisors prevent the product $g_1\cdots g_n$ from growing independently of their least common multiple $\mathrm{lcm}(g_1,\ldots,g_n)$. More precisely, the lemma shows that the product $g_1\cdots g_n$ is controlled relative to $\mathrm{lcm}(g_1,\ldots,g_n)$ up to an explicit factor depending only on $N$ and $n$. 

We next apply the preceding divisibility principle to arithmetic progressions of rational numbers. The problem reduces to studying the interaction between the numerators and a fixed global denominator. The following lemma shows that along any short segment of an arithmetic progression, the associated greatest common divisors cannot simultaneously be large.

\begin{lemma}\label{Sec5_lem2}
    Let 
    \[\{a + kb\:|\:0 \leq k < N\}\]
    be an arithmetic progression with $a,b \in \bbq$. Write 
    \[a = \frac{v_0}{u_0},\: b=\frac{v_1}{u_1} \quad \text{where} \quad \gcd(u_0,v_0)=\gcd(u_1,v_1)=1,\]
    and denote $s = \mathrm{lcm}(u_0,u_1)$. Write 
    \[r_k := a+kb = \frac{x_k}{s},\quad 0 \leq k< N\]
    and define 
    \[g_k := \gcd(x_k,s).\]
    Fix $\ell \geq 0$. Then there exists an effective constant $L(n)$ depending on $n$ such that 
    \[g_{\ell+1} \cdots g_{\ell+n} \mid L(n) \cdot s.\]
\end{lemma}
\begin{proof}
    Let $s=u_0u_0'=u_1u_1'$. Then $x_k = v_0u_0'+kv_1u_1'$. Note that $\gcd(v_1u_1',s) = \gcd(v_1u_1',u_1u_1') = u_1'$. 

    We will first prove 
    \begin{equation}\label{lem2_eq1}
        \gcd(g_{\ell+i},g_{\ell+j}) \mid (j-i),\quad 1 \leq i<j \leq n
    \end{equation}
    Fix $1 \leq i<j \leq n$ and let $h = \gcd(g_{\ell+i},g_{\ell+j})$. Then $h$ divides $x_{\ell+i}$, $x_{\ell+j}$, and $s$. Since $x_{\ell+j}-x_{\ell+i} = (j-i)v_1u_1'$, $h$ divides $(j-i)v_1u_1'$. From $\gcd(v_1u_1',s) = u_1'$, $h$ divides $(j-i)u_1'$. Assume there exists a prime $p$ such that $p \mid h$ and $p \mid u_1'$. Then $p \mid x_{\ell+i}$ and $p \mid u_1'$ imply $p \mid v_0u_0'$. However, $\gcd(u_1',v_0u_0') = 1$ because $\gcd(u_0,v_0)=1$ and $\gcd(u_0',u_1')=1$ (since $s=\mathrm{lcm}(u_0,u_1)$). Therefore, $(h,u_1')=1$. It follows that $h \mid (j-i)$.

    Now Lemma~\ref{Sec5_lem1} with \eqref{lem2_eq1} gives 
    \[g_{\ell+1} \cdots g_{\ell+n} \mid L(n) \cdot \mathrm{lcm}(g_{\ell+1},\ldots,g_{\ell+n}) \mid L(n) \cdot s.\]
    with $L(n) = \prod_{k=1}^{n-1}k!$.
\end{proof}

The following corollary shows that any dense subset of an arithmetic progression contains many elements satisfying the required gcd bound. These elements will later form the subset to which the gap principles may be applied.

\begin{corollary}\label{d=1_lemma}
    Suppose we are in Lemma~\ref{Sec5_lem2}. Let $0<\rho \leq 1$, $0<\delta<1$ be given, and set $m = \lceil 4/\delta\rho \rceil$. Let 
    \[H \subseteq G := \{a + kb\:|\:0 \leq k < N\}\]
    be a subset satisfying 
    \[\lvert H \rvert \geq \rho \lvert G \rvert.\]
    Then there exist effective constants $L(\delta,\rho)$ and $K(\delta,\rho)$ depending on $\delta$ and $\rho$ such that 
    \begin{equation}\label{lem3_eq1}
        \lvert \{r_k \in H\:|\:g_k \leq s^\delta\} \rvert \geq \frac{\rho}{2}\lvert G \rvert \quad \text{whenever} \quad s \geq L(\delta,\rho) \quad \text{and} \quad N \geq K(\delta,\rho).
    \end{equation}
\end{corollary}
\begin{proof}
    By substituting $n=2m$ in Lemma~\ref{Sec5_lem2}, we obtain a constant $L(\delta,\rho)$ so that 
    \begin{equation}\label{lem3_eq2}
        g_{\ell+1} \cdots g_{\ell+2m} \mid L(\delta,\rho) \cdot s
    \end{equation}
    for any $\ell \geq 0$. Take $K(\delta,\rho)$ by 
    \begin{equation}\label{lem3_eq3}
        K(\delta,\rho) = \frac{8m}{\rho}.
    \end{equation}
    Let $s \geq L(\delta,\rho)$ and $N \geq K(\delta,\rho)$.
    
    Assume $2m$ consecutive terms $r_{\ell+1},\ldots,r_{\ell+2m}$ are given and assume $g_{\ell+i} > s^\delta$ for $r$ numbers of $1 \leq i \leq 2m$. Then \eqref{lem3_eq2} and $s \geq L(\delta,\rho)$ implies 
    \[s^{\delta r} < g_{\ell+1} \cdots g_{\ell+2m} \leq L(\delta,\rho)s \leq s^2.\]
    This forces 
    \[r < \frac{2}{\delta} \leq \frac{m\rho}{2}.\]

    Now for each $0 \leq k \leq \left\lfloor \frac{N}{2m} \right\rfloor-1$, among 
    \[r_{2mk+1},\ldots,r_{2mk+2m},\]
    the number of $r_{2mk+i}$ such that $g_{2mk+i} > s^\delta$ is $<m\rho/2$. It follows that the number of $r_k \in G$ such that $g_k > s^\delta$ is 
    \[<\frac{m\rho}{2}\left\lfloor \frac{N}{2m} \right\rfloor + 2m \leq \frac{N\rho}{4}+2m \leq \frac{N\rho}{2},\] 
    where in the last line, we used $N \geq K(\delta,\rho)$ and \eqref{lem3_eq3}. Since $\lvert H \rvert \geq \rho N$, \eqref{lem3_eq1} is proved.
\end{proof}

We now extend the above extraction argument from ordinary arithmetic progressions to generalized arithmetic progressions. The higher dimension case is obtained by an induction on the dimension.

\begin{lemma}\label{Sec5_lem4}
    Let 
    \[\{a_0 + k_1a_1 + \cdots + k_da_d\:|\:0 \leq k_1 < N_1,\ldots,0 \leq k_d < N_d\}\]
    be a proper generalized arithmetic progression with $a_0,\ldots,a_d \in \bbq$. Write 
    \[a_i = \frac{v_i}{u_i} \quad \text{where} \quad \gcd(u_i,v_i)=1,\quad 0 \leq i \leq d,\]
    and denote $s = \mathrm{lcm}(u_0,\ldots,u_d)$. Write 
    \[r_{k_1,\ldots,k_d} := a_0 + k_1a_1 + \cdots + k_da_d = \frac{x_{k_1,\ldots,k_d}}{s},\quad 0 \leq k_1 < N_1,\ldots,0 \leq k_d < N_d\]
    and define 
    \[g_{k_1,\ldots,k_d} := \gcd(x_{k_1,\ldots,k_d},s).\] 
    Fix $\ell_1 \geq 0,\ldots,\ell_d \geq 0$. Then there exists an effective constant $L(n,d)$ depending on $n$ and $d$ such that 
    \begin{equation}\label{lem4_eq1}
        \prod_{0 \leq i_1 < n,\ldots,0 \leq i_d <n} g_{\ell_1+i_1,\ldots,\ell_d+i_d} \mid L(n,d)s^{dn^{d-1}}.
    \end{equation}
\end{lemma}
\begin{proof}
    Let $s=u_0u_0' = \cdots u_du_d'$. Then $x_{k_1,\ldots,k_d} = v_0u_0' + k_1v_1u_1' + \cdots k_dv_du_d'$. Note that $\gcd(v_du_d',s)=u_d'$.

    We will use induction on $d$. For $d=1$, this is Lemma~\ref{Sec5_lem2}. Suppose the theorem is proved for $d-1$. We denote $t = \mathrm{lcm}(u_0,\ldots,u_{d-1})$ and write 
    \[r_{k_1,\ldots,k_{d-1}} := a_0 + k_1a_1 + \cdots + k_{d-1}a_{d-1} = \frac{y_{k_1,\ldots,k_{d-1}}}{t},\quad 0 \leq k_1 < N_1,\ldots,0 \leq k_{d-1} < N_{d-1}\]
    and define 
    \[g_{k_1,\ldots,k_{d-1}} := \gcd(y_{k_1,\ldots,k_{d-1}},t).\] 
    Then the induction hypothesis gives 
    \begin{equation}\label{lem4_eq2}
        \prod_{0 \leq i_1 < n,\ldots,0 \leq i_{d-1} <n} g_{\ell_1+i_1,\ldots,\ell_{d-1}+i_{d-1}} \mid L(n,d-1)t^{(d-1)n^{d-2}}.
    \end{equation}

    Fix $0 \leq j_1 < n,\ldots,0 \leq j_{d-1} < n$. Consider the arithmetic progression 
    \[r_{\ell_1+j_1,\ldots,\ell_{d-1}+j_{d-1},\ell_d},r_{\ell_1+j_1,\ldots,\ell_{d-1}+j_{d-1},\ell_d+1},\ldots,r_{\ell_1+j_1,\ldots,\ell_{d-1}+j_{d-1},\ell_d+n-1}.\]
    For short, we write 
    \[r_i = r_{\ell_1+j_1,\ldots,\ell_{d-1}+j_{d-1},\ell_d+i},\quad x_i = x_{\ell_1+j_1,\ldots,\ell_{d-1}+j_{d-1},\ell_d+i},\quad g_i = g_{\ell_1+j_1,\ldots,\ell_{d-1}+j_{d-1},\ell_d+i}\]
    for $0 \leq i < n$.

    We will first prove 
    \begin{equation}\label{lem4_eq3}
        \gcd(g_i,g_j) \mid (j-i)g_{\ell_1+j_1,\ldots,\ell_{d-1}+j_{d-1}},\quad 1 \leq i<j \leq n.
    \end{equation}
    Fix $1 \leq i<j \leq n$ and let $h = \gcd(g_i,g_j)$. Then $h$ divides $x_i$, $x_j$, and $s$. Since $x_j-x_i = (j-i)v_du_d'$, $h$ divides $(j-i)v_du_d'$. From $\gcd(v_du_d',s) = u_d'$, $h$ divides $(j-i)u_d'$. Suppose an integer $k$ divides $x_i$, $u_d'$, and $s$. Then $k$ divides $x_{\ell_1+j_1,\ldots,\ell_{d-1}+j_{d-1},0}$, so that $k$ divides $g_{\ell_1+j_1,\ldots,\ell_{d-1}+j_{d-1},0}$. From 
    \[\frac{x_{\ell_1+j_1,\ldots,\ell_{d-1}+j_{d-1},0}}{s} = \frac{y_{\ell_1+j_1,\ldots,\ell_{d-1}+j_{d-1}}}{t},\]
    we have 
    \[g_{\ell_1+j_1,\ldots,\ell_{d-1}+j_{d-1},0} = g_{\ell_1+j_1,\ldots,\ell_{d-1}+j_{d-1}} \cdot \frac{s}{t}.\]
    Since $\mathrm{lcm}(t,u_d)=s$, $\gcd(u_d',s/t)=1$. Thus $k$ cannot divide $s/t$ and so $k$ divides $g_{\ell_1+j_1,\ldots,\ell_{d-1}+j_{d-1}}$. It follows that $h$ divides $(j-i)g_{\ell_1+j_1,\ldots,\ell_{d-1}+j_{d-1}}$. 

    By Lemma~\ref{Sec5_lem1} with \eqref{lem4_eq3}, 
    \[g_1 \cdots g_n \mid (g_{\ell_1+j_1,\ldots,\ell_{d-1}+j_{d-1}})^{n-1} \prod_{k=1}^{n-1}k! \cdot \mathrm{lcm}(g_1,\ldots,g_n) \mid (g_{\ell_1+j_1,\ldots,\ell_{d-1}+j_{d-1}})^{n-1} L(n) \cdot s.\]
    
    Now multiplying over all $0 \leq j_1 < n,\ldots,0 \leq j_{d-1} < n$ and applying \eqref{lem4_eq2}, we obtain 
    \[\prod_{0 \leq i_1 < n,\ldots,0 \leq i_d <n} g_{\ell_1+i_1,\ldots,\ell_d+i_d} \mid L(n)^{n^{d-1}}s^{n^{d-1}} (L(n,d-1)t^{(d-1)n^{d-2}})^{n-1}.\]
    By using $t \mid s$, we can bound the right side by $L(n,d)s^{dn^{d-1}}$, with 
    \[L(n,d) := L(n)^{n^{d-1}}L(n,d-1)^{n-1}.\]
    This proves \eqref{lem4_eq1}.
\end{proof}

We now extend Corollary~\ref{d=1_lemma} to generalized arithmetic progressions of higher dimension.

\begin{corollary}\label{general_d_lemma}
    Suppose we are in Lemma~\ref{Sec5_lem4}. Let $0<\rho \leq 1$, $0<\delta<1$ be given, and set $m = \lceil 4/\delta\rho \rceil$. Let 
    \[H \subseteq G := \{a_0 + k_1a_1 + \cdots + k_da_d\:|\:0 \leq k_1 < N_1,\ldots,0 \leq k_d < N_d\}\]
    be a subset satisfying 
    \[\lvert H \rvert \geq \rho \lvert G \rvert.\]
    Then there exist effective constants $L(\delta,\rho,d)$ and $K(\delta,\rho,d)$ depending on $\delta$, $\rho$, and $d$ such that 
    \begin{equation}\label{lem5_eq1}
        \begin{aligned}
            &\lvert \{r_{k_1,\ldots,k_d} \in H\:|\:g_{k_1,\ldots,k_d} \leq s^\delta\} \rvert \geq \frac{\rho}{2}\lvert G \rvert \\
            &\text{whenever} \quad s \geq L(\delta,\rho,d) \quad \text{and} \quad N_1,\ldots,N_d \geq K(\delta,\rho,d).
        \end{aligned}
    \end{equation}
\end{corollary}
\begin{proof}
    By substituting $n=2md$ in Lemma~\ref{Sec5_lem4}, we obtain an effective constant $L(\delta,\rho,d)$ such that 
    \begin{equation}\label{lem5_eq2}
        \prod_{0 \leq i_1 < 2md,\ldots,0 \leq i_d <2md} g_{\ell_1+i_1,\ldots,\ell_d+i_d} \mid L(\delta,\rho,d)s^{d(2md)^{d-1}}
    \end{equation}
    for any $\ell_1 \geq 0,\ldots,\ell_d \geq 0$. Take a constant $K(\delta,\rho,d)$ by 
    \begin{equation}\label{lem5_eq3}
        K(\delta,\rho,d) = \frac{8md^2}{\rho}.
    \end{equation}
    Let $s \geq L(\delta,\rho,d)$ and $N_1,\ldots,N_d \geq K(\delta,\rho,d)$.

    Assume $(2md)^d$ consecutive terms $r_{\ell_1+i_1,\ldots,\ell_d+i_d}$, $0 \leq i_1 < 2md,\ldots,0 \leq i_d < 2md$ are given and assume $g_{\ell_1+i_1,\ldots,\ell_d+i_d} > s^\delta$ for $r$ numbers of $0 \leq i_1 < 2md,\ldots,0 \leq i_d < 2md$. Then \eqref{lem5_eq2} and $s \geq L(\delta,\rho,d)$ imply 
    \[s^{r\delta} < \prod_{0 \leq i_1 < 2md,\ldots,0 \leq i_d <2md} g_{\ell_1+i_1,\ldots,\ell_d+i_d} \leq L(\delta,\rho,d)s^{d(2md)^{d-1}} \leq s^{2d(2md)^{d-1}}.\]
    This forces 
    \[r < \frac{2d(2md)^{d-1}}{\delta} \leq \frac{(2md)^d\rho}{4}.\]

    Now for each $0 \leq k_1 \leq \left\lfloor \frac{N_1}{2md} \right\rfloor-1,\ldots,0 \leq k_d \leq \left\lfloor \frac{N_d}{2md} \right\rfloor-1$, among 
    \[r_{\ell_1+i_1,\ldots,\ell_d+i_d},\quad 0 \leq i_1 < 2md,\ldots,0 \leq i_d < 2md,\]
    the number of $r_{\ell_1+i_1,\ldots,\ell_d+i_d}$ such that $g_{\ell_1+i_1,\ldots,\ell_d+i_d}>s^\delta$ is $<(2md)^d\rho/4$. It follows that the number of $r_k \in G$ such that $g_k > s^\delta$ is 
    \[< \frac{(2md)^d\rho}{4}\left\lfloor\frac{N_1}{2md} \right\rfloor \cdots \left\lfloor\frac{N_d}{2md} \right\rfloor + 2md\left(\frac{1}{N_1} + \cdots + \frac{1}{N_d}\right)\lvert G \rvert \leq \frac{\rho}{2}\lvert G \rvert,\]
    where in the last line, we used $N_1,\ldots,N_d \geq K(\delta,\rho,d)$ and \eqref{lem5_eq3}. Since $\lvert H \rvert \geq \rho \lvert G \rvert$, \eqref{lem5_eq1} is proved.
\end{proof}

Combining the preceding results, we conclude that dense subsets of proper generalized arithmetic progressions contain a positive proportion of elements satisfying the required primitiveness condition. In particular, whenever a family of rational points occupies a positive proportion of a proper generalized arithmetic progression, one may extract a large subset to which the gap principles of Section~\ref{Gap_principles} apply.

\section{Reduction of Theorem~\ref{main_theorem}}\label{Reduction}

In this section we reduce Theorem~\ref{main_theorem} to two separate statements according to the size of the $x$-coordinates. 

Recall that $E$ is given by a short Weierstrass equation
\[E : y^2 = x^3 + Ax + B,\]
and that $X = \max\{|A|^3, |B|^2\}$.
If $(x,y) \in E(\bbq)$ and $x < -2X^{1/6}$, then $x^3 + Ax + B < 0$, which contradicts $y^2 \geq 0$. Hence every rational point satisfies either
\[|x| \leq 2X^{1/6} \quad \text{or} \quad x \geq X^{1/6}.\]

Accordingly, we decompose any finite subset $\mathcal P \subseteq E(\bbq)$ into its \emph{small $x$} part and its \emph{large $x$} part, and treat these two regimes separately.
This reduction leads to the following two theorems.

\begin{theorem}\label{main_theorem_small_x}
    Let $E/\bbq$ be an elliptic curve of Mordell-Weil rank $r \geq 1$. Let $d \geq 1$ be an integer and let $0<\rho \leq 1$. Then there exists an effectively computable constant $A(E,d,\rho)>0$ with the following property. 
    
    For any finite subset $\mathcal{P} \subseteq E(\bbq)$ such that
    \begin{enumerate}
        \item $\lvert x(P) \rvert \le 2X^{1/6}$ for all $P \in \mathcal{P}$,
        \item the set of $x$-coordinates $x(\mathcal{P})$ is contained in a $d$-dimensional proper generalized arithmetic progression $G$, and
        \item $\lvert x(\mathcal{P}) \rvert \geq \rho \lvert G \rvert$,
    \end{enumerate}
    we have 
    \[\lvert \mathcal{P} \rvert \leq A(E,d,\rho)^r.\]
    
    Moreover, assume that Conjecture~\ref{Lang_conjecture} holds with an explicitly given admissible constant $c_L>0$. Then $A(E,d,\rho)$ may be taken to be an effectively computable constant depending only on $d$ and $\rho$.
\end{theorem}

\begin{theorem}\label{main_theorem_large_x}
    Let $E/\bbq$ be an elliptic curve of Mordell-Weil rank $r \geq 1$. Let $d \geq 1$ be an integer and let $0<\rho \leq 1$. Then there exists an effectively computable constant $A(E,d,\rho)>0$ with the following property. 
    
    For any finite subset $\mathcal{P} \subseteq E(\bbq)$ such that
    \begin{enumerate}
        \item $x(P) \geq X^{1/6}$ for all $P \in \mathcal{P}$,
        \item the set of $x$-coordinates $x(\mathcal{P})$ is contained in a $d$-dimensional proper generalized arithmetic progression $G$, and
        \item $\lvert x(\mathcal{P}) \rvert \geq \rho \lvert G \rvert$,
    \end{enumerate}
    we have 
    \[\lvert\mathcal{P} \rvert \leq A(E,d,\rho)^r.\]
    
    Moreover, assume that Conjecture~\ref{Lang_conjecture} holds with an explicitly given admissible constant $c_L>0$. Then $A(E,d,\rho)$ may be taken to be an effectively computable constant depending only on $d$ and $\rho$.
\end{theorem}

We now explain how Theorems~\ref{main_theorem_small_x} and~\ref{main_theorem_large_x} together imply Theorem~\ref{main_theorem}.

\begin{proof}[Proof of Theorem~\ref{main_theorem} assuming Theorems~\ref{main_theorem_small_x} and~\ref{main_theorem_large_x}]
    Suppose we are given a finite subset $\mathcal{P} \subseteq E(\bbq)$ such that
    \begin{enumerate}
        \item the set of $x$-coordinates $x(\mathcal{P})$ is contained in a $d$-dimensional generalized arithmetic progression $G$, and
        \item $\lvert x(\mathcal{P}) \rvert \geq \rho \lvert G \rvert$.
    \end{enumerate}
    
    First, define 
    \[\mathcal{P}' := \{P \in E(\bbq)\:|\:x(P) \in x(\mathcal{P})\},\quad \mathcal{P}'' = \{P \in E(\bbq)\:|\:x(P) \in x(\mathcal{P}),\:y(P) \geq 0\}.\]
    Then 
    \[\lvert \mathcal{P} \rvert \leq \lvert \mathcal{P}' \rvert \leq 2\lvert \mathcal{P}'' \rvert\]
    while 
    \[x(\mathcal{P}) = x(\mathcal{P}') = x(\mathcal{P}'').\]
    Replacing $\mathcal{P}$ by $\mathcal{P}''$, we may therefore assume that $y(P) \geq 0$ for every $P \in \mathcal{P}$. This implies $\lvert \mathcal{S} \rvert = \lvert x(\mathcal{S}) \rvert$ for any subset $\mathcal{S} \subseteq \mathcal{P}$.
    
    We now decompose 
    \[\mathcal{P} = \mathcal{P}_{small} \cup \mathcal{P}_{large}\]
    where 
    \[\mathcal{P}_{small} := \{P \in \mathcal{P}\:|\:\lvert x(P) \rvert \le 2X^{1/6}\},\quad \mathcal{P}_{large} := \{P \in \mathcal{P}\:|\:x(P) \geq X^{1/6}\}.\]
    Then we have 
    \[\lvert \mathcal{P} \rvert \leq \lvert \mathcal{P} \rvert_{small}+\lvert \mathcal{P}_{large} \rvert \leq 2\max\{\lvert \mathcal{P} \rvert_{small},\lvert \mathcal{P}_{large} \rvert\}\]
    and 
    \[\max\{\lvert x(\mathcal{P}_{small}) \rvert,\lvert x(\mathcal{P}_{large}) \rvert\} \geq \frac{\rho}{2}\lvert G \rvert.\]
    Now apply Theorem~\ref{main_theorem_small_x} or Theorem~\ref{main_theorem_large_x}.
\end{proof}

We conclude this section with a simple counting lemma that bounds rational points of small canonical height. Such estimates appear frequently in the literature; see, for example, \cite{Sil87}[Lemma~1.2].

\begin{lemma}\label{Small_points}
    Let $E/\bbq$ be an elliptic curve of Mordell-Weil rank $r \geq 1$. Let $M>0$ be a fixed constant and let 
    \[\mathcal{S}_M := \{P \in E(\bbq)\:|\:\hat{h}(P) \leq M \log X\}.\] 
    Then there exists an effectively computable constant $A(E,M)>0$ with 
    \[\lvert \mathcal{S}_M \rvert \leq A(E,M)^r.\]

    Moreover, assume that Conjecture~\ref{Lang_conjecture} holds with an explicitly given admissible constant $c_L>0$. Then $A(E,M)$ may be taken to be an effectively computable constant depending only on $M$.
\end{lemma}
\begin{proof}
    By \cite{Sil81}, let $c(E)>0$ be an effective constant such that 
    \[\hat{h}(P) \geq c(E)\log X\]
    for all non-torsion points $P \in E(\bbq)$. Define 
    \[N = N(E,M) := \left\lceil \sqrt{3M/c(E)} \right\rceil.\]
    Note that if we assume Conjecture~\ref{Lang_conjecture}, then $c(E)$ can be chosen to be a constant multiple of explicit absolute constant $c_L$, so that $N=N(E,M)$ is an effectively computable constant depending only on $M$.

    Fix $R \in E(\bbq)$ and define 
    \[\mathcal{S}_M(R) := \{P \in \mathcal{S}_M\:|\:P-R \in NE(\bbq)\}.\]
    Let $\{P_1,\ldots,P_n\} \subseteq \mathcal{S}_M(R)$ be maximal with the property that $P_i-P_j$ is non-torsion whenever $i \neq j$. For $i \neq j$, write $P_i-P_j = NS$ for some non-torsion point $S$. From 
    \[\hat{h}(S) \geq c(E)\log X,\]
    we have 
    \[\hat{h}(P_i-P_j) = N^2\hat{h}(S) \geq N^2c(E) \log X \geq 3M\log X.\] 
    Therefore, 
    \begin{align*}
        \cos \theta_{P_i,P_j} &= \frac{\hat{h}(P_i) + \hat{h}(P_j) - \hat{h}(P_i-P_j)}{2\sqrt{\hat{h}(P_i)}\sqrt{\hat{h}(P_j)}} \leq -\frac{M\log X}{2\sqrt{\hat{h}(P_i)}\sqrt{\hat{h}(P_j)}} \leq -\frac{M\log X}{2M\log X} = -\frac{1}{2}<0.
    \end{align*}
    By Theorem~\ref{spherical_code_bound2}, $n$ is bounded by an absolute constant $C$. By maximality, every $P \in \mathcal{S}_M(R)$ differs from some $P_i$ by a torsion point. Since  $\lvert E(\bbq)_{tors} \rvert \leq 16$ by Mazur's torsion theorem, we obtain 
    \[\lvert \mathcal{S}_M(R) \rvert \leq 16n \leq 16C.\]

    Since there are $N^r$ cosets of $NE(\bbq)$ in $E(\bbq)$, we conclude that 
    \[\lvert \mathcal{S}_M \rvert \leq 16C \cdot N^r,\]
    which proves the lemma.
\end{proof}

\section{Proof of Theorem~\ref{main_theorem_small_x}}\label{Proof1}

In this section, we prove Theorem~\ref{main_theorem_small_x} by induction on the dimension $d$ of the generalized arithmetic progression. Let $E/\bbq$, $d \geq 1$, and $0<\rho \leq 1$ in Theorem~\ref{main_theorem_small_x} be given, and suppose that the theorem has been proved in all smaller dimensions $e<d$.

\subsection*{Choice of parameters}

We first choose several auxiliary parameters that will be used throughout the proof. Fix an absolute constant $\delta=0.1$. Take effective constants $L(\delta,\rho,d)$ and $K(\delta,\rho,d)$ in Lemma~\ref{general_d_lemma}.

We then choose effective $\gamma>0$ and $0<\theta_1 < \pi/2$ satisfying 
\begin{equation}\label{Sec6_eq2}
    \gamma^{-1} \geq \frac{\log L(\delta,\rho,d)}{17}
\end{equation}
and 
\begin{equation}\label{Sec6_eq3}
    \frac{1+2\delta}{2(1-\delta-\gamma)} + \frac{2}{(1-\delta)\gamma^{-1}-1} \leq \cos \theta_1.
\end{equation}
Note that $\gamma$ and $\theta_1$ depend only on $d$ and $\rho$.

Next, by Lemma~\ref{spherical_code_bound1} and Lemma~\ref{Small_points}, there exists an effective constant $B(E,d,\rho)>0$ such that 
\begin{equation}\label{Sec6_eq6}
    A(r,\theta_1) \leq B(E,d,\rho)^r
\end{equation}
and 
\begin{equation}\label{Sec6_eq7}
    \lvert \{P \in E(\bbq)\:|\:\hat{h}(P) \leq (\gamma^{-1}+1)\log X\} \rvert \leq B(E,d,\rho)^r.
\end{equation}
Note that if we assume Conjecture~\ref{Lang_conjecture}, then $B(E,d,\rho)$ may be chosen to be an effectively computable constant depending only on $d$ and $\rho$.

\subsection*{Settings}

Suppose we are given a finite subset $\mathcal{P} \subseteq E(\bbq)$ such that
\begin{enumerate}
    \item $\lvert x(P) \rvert \le 2X^{1/6}$ for all $P \in \mathcal{P}$,
    \item the set of $x$-coordinates $x(\mathcal{P})$ is contained in a $d$-dimensional proper generalized arithmetic progression $G$, and
    \item $\lvert x(\mathcal{P}) \rvert \geq \rho \lvert G \rvert$.
\end{enumerate}
We will prove the existence of an effective constant $A(E,d,\rho)>0$ such that 
\[\lvert \mathcal{P} \rvert \leq A(E,d,\rho)^r.\]

As in the proof of Theorem~\ref{main_theorem} assuming Theorems~\ref{main_theorem_small_x} and~\ref{main_theorem_large_x}, we may assume that $y(P) \geq 0$ for every $P \in \mathcal{P}$. Again this implies $\lvert \mathcal{S} \rvert = \lvert x(\mathcal{S}) \rvert$ for any subset $\mathcal{S} \subseteq \mathcal{P}$.

Write 
\[G := \{a_0 + k_1a_1 + \cdots + k_da_d\:|\:0 \leq k_1 < N_1,\ldots,0 \leq k_d < N_d\}\]
with $a_0 \in \bbq$, $a_1,\ldots,a_d \in \bbq_+$, write 
\[a_i = \frac{v_i}{u_i} \quad \text{where} \quad \gcd(u_i,v_i)=1,\quad 0 \leq i \leq d,\]
and denote $s = \mathrm{lcm}(u_0,\ldots,u_d)$. Write 
\[r_{k_1,\ldots,k_d} := a_0 + k_1a_1 + \cdots + k_da_d = \frac{x_{k_1,\ldots,k_d}}{s},\quad 0 \leq k_1 < N_1,\ldots,0 \leq k_d < N_d\]
and define 
\[g_{k_1,\ldots,k_d} := \gcd(x_{k_1,\ldots,k_d},s).\] 

By Lemma~\ref{general_d_lemma}, we have 
\begin{equation}\label{Sec6_eq1}
    \begin{aligned}
    &\lvert \{r_{k_1,\ldots,k_d} \in x(\mathcal{P})\:|\:g_{k_1,\ldots,k_d} \leq s^\delta\} \rvert \geq \frac{\rho}{2}\lvert G \rvert \\
    &\text{whenever} \quad s \geq L(\delta,\rho,d) \quad \text{and} \quad N_1,\ldots,N_d \geq K(\delta,\rho,d).
\end{aligned}
\end{equation}

\subsection*{Reduction to the case $N_1,\ldots,N_d \geq K(\delta,\rho,d)$}

We first prove that it suffices to assume $N_1,\ldots,N_d \geq K(\delta,\rho,d)$. Suppose there exists some $1 \leq i \leq d$ such that $N_i < K(\delta,\rho,d)$. After reordering the indices, assume that $N_1,\ldots,N_e \geq K(\delta,\rho,d)$ and $N_{e+1},\ldots,N_d < K(\delta,\rho,d)$ for some $0 \leq e < d$. 

For each $0 \leq j_{e+1} < N_{e+1},\ldots,0 \leq j_d < N_d$, let 
\[G_{j_{e+1},\ldots,j_d} := \{(a_0+j_{e+1}a_{e+1} + \cdots + j_da_d) + k_1a_1 + \cdots + k_ea_e\:|\:0 \leq k_1 < N_1,\ldots,0 \leq k_e < N_e\}\]
be an $e$-dimensional generalized arithmetic progression and let 
\[\mathcal{P}_{j_{e+1},\ldots,j_d} := \{P \in \mathcal{P}\:|\:x(P) \in G_{j_{e+1},\ldots,j_d}\}.\]
Take $0 \leq \ell_{e+1} < N_{e+1},\ldots,0 \leq \ell_d < N_d$ so that 
\[\max_{0 \leq j_{e+1} < N_{e+1},\ldots,0 \leq j_d < N_d} \lvert \mathcal{P}_{j_{e+1},\ldots,j_d} \rvert = \lvert \mathcal{P}_{\ell_{e+1},\ldots,\ell_d} \rvert\]
By the pigeonhole principle and the maximality, we must have 
\[\lvert x(\mathcal{P}_{\ell_{e+1},\ldots,\ell_d}) \rvert \geq \rho \lvert G_{\ell_{e+1},\ldots,\ell_d} \rvert.\]

Since $e<d$, by the induction hypothesis, there exists an effective constant $A(E,e,\rho)>0$ such that 
\[\lvert \mathcal{P}_{\ell_{e+1},\ldots,\ell_d} \rvert \leq A(E,e,\rho)^r.\]
Now we have 
\[\lvert \mathcal{P} \rvert \leq N_{e+1} \cdots N_d \lvert \mathcal{P}_{\ell_{e+1},\ldots,\ell_d} \rvert \leq K(\delta,\rho,d)^{d-e}A(E,e,\rho)^r.\]
Letting 
\[A(E,d,\rho) := \max_{0 \leq e < d} K(\delta,\rho,d)^{d-e}A(E,e,\rho),\]
we obtain 
\[\lvert \mathcal{P} \rvert \leq A(E,d,\rho)^r.\]

Hence, we will assume $N_1,\ldots,N_d \geq K(\delta,\rho,d)$ in the below proof.

\subsection{When $h(s) \leq \gamma^{-1}\log X$}

We first treat the case where the denominator $s$ is relatively small.

Let $\mathcal{P} = \{P_1,\ldots,P_n\}$ and write 
\[x(P_i) = \frac{y_i}{s},\quad 1 \leq i \leq n.\]
From $\lvert x(P_i) \rvert \leq 2X^{1/6}$, we have 
\[h(P_i) \leq \max\{h(y_i),h(s)\} \leq h(s) + \frac{1}{3}\log X \leq \left(\gamma^{-1} + \frac{1}{3} \right)\log X,\quad 1 \leq i \leq n.\]
By Lemma~\ref{Weil_canonical_height}, 
\[\hat{h}(P_i) \leq (\gamma^{-1}+1)\log X,\quad 1 \leq i \leq n.\]
By the choice \eqref{Sec6_eq7}, 
\[\lvert \mathcal{P} \rvert = n \leq B(E,d,\rho)^r.\]
Letting 
\[A(E,d,\rho) := B(E,d,\rho)\]
ends the proof in this case.

\subsection{When $h(s) > \gamma^{-1}\log X$}

We now consider the complementary case where the denominator $s$ is large.

By \eqref{X_canonical_lower_bound} and \eqref{Sec6_eq2}, we have 
\[h(s) > 17\gamma^{-1} \geq \log L(\delta,\rho,d).\]
So we have $s \geq L(\delta,\rho,d)$. Also we have $N_1,\ldots,N_d \geq K(\delta,\rho,d)$ by our previous argument. Therefore, \eqref{Sec6_eq1} implies 
\begin{equation}\label{Sec6_eq4}
    \lvert \{r_{k_1,\ldots,k_d} \in x(\mathcal{P})\:|\:g_{k_1,\ldots,k_d} \leq s^\delta\} \rvert \geq \frac{\rho}{2}\lvert G \rvert.
\end{equation}

Let 
\[\{P \in \mathcal{P}\:|\:x(P) = r_{k_1,\ldots,k_d},\:g_{k_1,\ldots,k_d} \leq s^\delta\} = \{P_1,\ldots,P_n\}\]
and write 
\[x(P_i) = \frac{y_i}{s},\quad 1 \leq i \leq n.\]
We will apply the gap principle Theorem~\ref{gap_principle2_large_s} for these rational points.

For each $i$, $\gcd(y_i,s) \leq s^\delta$ implies 
\[h(P_i) \geq (1-\delta)h(s) > (1-\delta)\gamma^{-1}\log X.\]
Then by Lemma~\ref{Weil_canonical_height}, 
\[\hat{h}(P_i) \geq ((1-\delta)\gamma^{-1}-1)\log X.\]
By taking $M=(1-\delta)\gamma^{-1}-1$ in Theorem~\ref{gap_principle2_large_s}, we obtain 
\[\cos \theta_{P_i,P_j} \leq \frac{1+2\delta}{2(1-\delta-\gamma)} + \frac{2}{(1-\delta)\gamma^{-1}-1} \leq \cos \theta_1\]
whenever $i \neq j$.

By the choice \eqref{Sec6_eq6}, 
\[n \leq B(E,d,\rho)^r.\]
Hence, \eqref{Sec6_eq4} implies 
\[\lvert \mathcal{P} \rvert \leq \lvert G \rvert \leq \frac{2}{\rho}n \leq \frac{2}{\rho}B(E,d,\rho)^r.\]
Letting 
\[A(E,d,\rho) := \frac{2}{\rho}B(E,d,\rho)\]
ends the proof in this case.

\section{Proof of Theorem~\ref{main_theorem_large_x}}\label{Proof2}

In this section, we prove Theorem~\ref{main_theorem_large_x} by induction on the dimension $d$ of the generalized arithmetic progression. Let $E/\bbq$, $d \geq 1$, and $0<\rho \leq 1$ in Theorem~\ref{main_theorem_large_x} be given, and suppose that the theorem has been proved in all smaller dimensions $e<d$.

\subsection*{Choice of parameters}

We first choose several auxiliary parameters that will be used throughout the proof. Fix an absolute constant $\delta=0.1$. Take effective constants $L(\delta,\rho,d)$ and $K(\delta,\rho,d)$ in Lemma~\ref{general_d_lemma}.

We then choose $\gamma>0$, $0<\theta_2 < \pi/2$, and $0<\theta_3<\pi/2$ satisfying 
\begin{equation}\label{Sec7_eq1}
    \gamma^{-1} \geq \frac{\log L(\delta,\rho,d)}{17},
\end{equation}
\begin{equation}\label{Sec7_eq2}
    \frac{1}{2}\sqrt{\frac{1+\gamma/10}{1-\gamma/10}\frac{9}{4}} + \frac{3}{20} + 0.4\gamma \leq \cos \theta_2,
\end{equation}
and 
\begin{equation}\label{Sec7_eq3}
    \frac{1}{2}\sqrt{\frac{1+\gamma/10}{1-\gamma/10}\frac{2}{1-\delta}} + \frac{3\delta}{2(1-\delta-\gamma)} + 0.4\gamma \leq \cos \theta_3.
\end{equation}
Note that $\gamma$, $\theta_2$, and $\theta_3$ depend only on $d$ and $\rho$.

Next, by Lemma~\ref{spherical_code_bound1} and Lemma~\ref{Small_points}, there exists an effective constant $B(E,d,\rho)>0$ such that 
\begin{equation}\label{Sec7_eq4}
    A(r,\theta_2) \leq B(E,d,\rho)^r,
\end{equation}
\begin{equation}\label{Sec7_eq5}
    A(r,\theta_3) \leq B(E,d,\rho)^r,
\end{equation}
and 
\begin{equation}\label{Sec7_eq6}
    \lvert \{P \in E(\bbq)\:|\:\hat{h}(P) \leq 10\gamma^{-1}\log X\} \rvert \leq B(E,d,\rho)^r.
\end{equation}
Note that if we assume Conjecture~\ref{Lang_conjecture}, then $B(E,d,\rho)$ may be chosen to be an effectively computable constant depending only on $d$ and $\rho$.

Now choose an integer $m$ satisfying 
\begin{equation}\label{Sec7_eq7}
    2B(E,d,\rho)^r < \frac{\rho}{4}m \leq 4B(E,d,\rho)^r.
\end{equation}
This $m$ will work for the contradiction argument. 

Finally, let 
\begin{equation}\label{Sec7_eq8}
    J(\rho,d) := 12d/\rho.
\end{equation}

\subsection*{Settings}

Suppose we are given a finite subset $\mathcal{P} \subseteq E(\bbq)$ such that
\begin{enumerate}
    \item $x(P) \geq X^{1/6}$ for all $P \in \mathcal{P}$,
    \item the set of $x$-coordinates $x(\mathcal{P})$ is contained in a $d$-dimensional proper generalized arithmetic progression $G$, and
    \item $\lvert x(\mathcal{P}) \rvert \geq \rho \lvert G \rvert$.
\end{enumerate}
We will prove the existence of a constant $A(E,d,\rho)>0$ such that 
\[\lvert \mathcal{P} \rvert \leq A(E,d,\rho)^r.\]

As in the proof of Theorem~\ref{main_theorem} assuming Theorems~\ref{main_theorem_small_x} and~\ref{main_theorem_large_x}, we may assume that $y(P) \geq 0$ for every $P \in \mathcal{P}$. Again this implies $\lvert \mathcal{S} \rvert = \lvert x(\mathcal{S}) \rvert$ for any subset $\mathcal{S} \subseteq \mathcal{P}$.

Write 
\[G := \{a_0 + k_1a_1 + \cdots + k_da_d\:|\:0 \leq k_1 < N_1,\ldots,0 \leq k_d < N_d\}\]
with $a_0 \in \bbq$, $a_1,\ldots,a_d \in \bbq_+$, write 
\[a_i = \frac{v_i}{u_i} \quad \text{where} \quad \gcd(u_i,v_i)=1,\quad 0 \leq i \leq d,\]
and denote $s = \mathrm{lcm}(u_0,\ldots,u_d)$. Write 
\[r_{k_1,\ldots,k_d} := a_0 + k_1a_1 + \cdots + k_da_d = \frac{x_{k_1,\ldots,k_d}}{s},\quad 0 \leq k_1 < N_1,\ldots,0 \leq k_d < N_d\]
and define 
\[g_{k_1,\ldots,k_d} := \gcd(x_{k_1,\ldots,k_d},s).\] 

By Lemma~\ref{general_d_lemma}, we have 
\begin{equation}\label{Sec7_eq9}
    \begin{aligned}
    &\lvert \{r_{k_1,\ldots,k_d} \in x(\mathcal{P})\:|\:g_{k_1,\ldots,k_d} \leq s^\delta\} \rvert \geq \frac{\rho}{2}\lvert G \rvert \\
    &\text{whenever} \quad s \geq L(\delta,\rho,d) \quad \text{and} \quad N_1,\ldots,N_d \geq K(\delta,\rho,d).
    \end{aligned}
\end{equation}
The same slicing argument as in Section~\ref{Proof1}, using the induction hypothesis, we may assume that $N_1,\ldots,N_d \geq K(\delta,\rho,d)$. 

Define $M_i = \lfloor N_i/m \rfloor$ for each $1 \leq i \leq d$. We first prove that if there exists some $1 \leq i \leq d$ such that $M_i < J(\rho,d)$, then there exists a constant $A(E,d,\rho)>0$ such that 
\[\lvert \mathcal{P} \rvert \leq A(E,d,\rho)^r.\]

\subsection*{Reduction to $M_1,\ldots,M_d \geq J(\rho,d)$}

Suppose there exists some $1 \leq i \leq d$ such that $M_i < J(\rho,d)$. After reordering the indices, assume that $M_1,\ldots,M_e \geq J(\rho,d)$ and $M_{e+1},\ldots,M_d < J(\rho,d)$ for some $0 \leq e < d$. 

For each $0 \leq j_{e+1} < N_{e+1},\ldots,0 \leq j_d < N_d$, let 
\[G_{j_{e+1},\ldots,j_d} := \{(a_0+j_{e+1}a_{e+1} + \cdots + j_da_d) + k_1a_1 + \cdots + k_ea_e\:|\:0 \leq k_1 < N_1,\ldots,0 \leq k_e < N_e\}\]
be an $e$-dimensional generalized arithmetic progression and let 
\[\mathcal{P}_{j_{e+1},\ldots,j_d} := \{P \in \mathcal{P}\:|\:x(P) \in G_{j_{e+1},\ldots,j_d}\}.\]
Take $0 \leq \ell_{e+1} < N_{e+1},\ldots,0 \leq \ell_d < N_d$ so that 
\[\max_{0 \leq j_{e+1} < N_{e+1},\ldots,0 \leq j_d < N_d} \lvert \mathcal{P}_{j_{e+1},\ldots,j_d} \rvert = \lvert \mathcal{P}_{\ell_{e+1},\ldots,\ell_d} \rvert\]
By the pigeonhole principle and the maximality, we must have 
\[\lvert x(\mathcal{P}_{\ell_{e+1},\ldots,\ell_d}) \rvert \geq \rho \lvert G_{\ell_{e+1},\ldots,\ell_d} \rvert.\]

Since $e<d$, by the induction hypothesis, there exists a constant $A(E,e,\rho)>0$ such that 
\[\lvert \mathcal{P}_{\ell_{e+1},\ldots,\ell_d} \rvert \leq A(E,e,\rho)^r.\]

For each $e+1 \leq i \leq d$, 
\[N_i \leq (M_i+1)m \leq (J(\rho,d)+1)m \leq \frac{300d}{\rho^2}B(E,d,\rho)^r.\]
Therefore, 
\[\lvert \mathcal{P} \rvert \leq N_{e+1} \cdots N_d \lvert \mathcal{P}_{\ell_{e+1},\ldots,\ell_d} \rvert \leq \left(\frac{300d}{\rho^2}B(E,d,\rho)^r\right)^{d-e}A(E,e,\rho)^r.\]
Letting 
\[A(E,d,\rho) := \max_{0 \leq e < d} \left(\frac{300d}{\rho^2}B(E,d,\rho)\right)^{d-e}A(E,e,\rho),\]
we obtain 
\[\lvert \mathcal{P} \rvert \leq A(E,d,\rho)^r.\]

We may therefore assume $M_1,\ldots,M_d \geq J(\rho,d)$. Now we will derive a contradiction by applying the gap principles Theorem~\ref{gap_principle1_large_s} and Theorem~\ref{gap_principle1_small_s}.

\subsection{When $h(s) \leq \gamma^{-1}\log X$}\label{largex_smalls}

Let $H = x(\mathcal{P})$. Recall that 
\begin{equation}\label{Sec7_eq10}
    \lvert H \rvert \geq \rho\lvert G \rvert.
\end{equation}

For each $0 \leq \ell_1 < M_1,\ldots,0 \leq \ell_d < M_d$, let 
\[G_{\ell_1,\ldots,\ell_d} := \{r_{k_1,\ldots,k_d}\:|\:m\ell_1 \leq k_1 < m(\ell_1+1),\ldots,m\ell_d \leq k_d < m(\ell_d+1)\}\]
and let 
\[H_{\ell_1,\ldots,\ell_d} := H \cap G_{\ell_1,\ldots,\ell_d}.\]
Then 
\begin{equation}\label{Sec7_eq11}
    \lvert H \rvert \leq \sum_{0 \leq \ell_1 < M_1,\ldots,0 \leq \ell_d < M_d} \lvert H_{\ell_1,\ldots,\ell_d} \rvert + m^d\left(\frac{1}{M_1} + \cdots + \frac{1}{M_d}\right)M_1 \cdots M_d
\end{equation}
Here the additional term accounts for the incomplete blocks near the boundary when $N_i$ is not a multiple of $m$.

Let 
\[\mathcal{L} := \{0,1,\ldots,M_1-1\} \times \cdots \times \{0,1,\ldots,M_d-1\}\]
be the index set. Define 
\[\mathcal{L}_1 := \{(\ell_1,\ldots,\ell_d) \in \mathcal{L}\:|\:r_{k_1,\ldots,k_d}<0\:\text{for all}\:r_{k_1,\ldots,k_d} \in G_{\ell_1,\ldots,\ell_d}\}\]
and
\[\mathcal{L}_2 := \{(\ell_1,\ldots,\ell_d) \in \mathcal{L}\:|\:r_{k_1,\ldots,k_d} \geq 0\:\text{for some}\:r_{k_1,\ldots,k_d} \in G_{\ell_1,\ldots,\ell_d}\}.\]
Note that if $(\ell_1,\ldots,\ell_d) \in \mathcal{L}_1$, then $H_{\ell_1,\ldots,\ell_d}$ is empty. Therefore, for counting, it suffices to consider $H_{\ell_1,\ldots,\ell_d}$ for $(\ell_1,\ldots,\ell_d) \in \mathcal{L}_2$.
Define 
\[\mathcal{L}_3 := \{(\ell_1,\ldots,\ell_d) \in \mathcal{L}_2\:|\:(\ell_1-1,\ldots,\ell_d-1) \in \mathcal{L}_2\}\]
and 
\[\mathcal{L}_4 := \{(\ell_1,\ldots,\ell_d) \in \mathcal{L}_2\:|\:(\ell_1-1,\ldots,\ell_d-1) \in \mathcal{L}_3\}.\]
It is clear that for every $(\ell_1,\ldots,\ell_d) \in \mathcal{L}_3$, 
\[r_{k_1,\ldots,k_d} \geq 0\quad\text{for all}\quad r_{k_1,\ldots,k_d} \in G_{\ell_1,\ldots,\ell_d}\]
and for every $(\ell_1,\ldots,\ell_d) \in \mathcal{L}_4$, 
\[r_{k_1,\ldots,k_d} \geq m(a_1 + \cdots + a_d) \quad\text{for all}\quad r_{k_1,\ldots,k_d} \in G_{\ell_1,\ldots,\ell_d}.\]
Indeed, since $a_1,\ldots,a_d>0$, the minimum (respectively maximum) of $r_{k_1,\ldots,k_d}$ over a block $G_{\ell_1,\ldots,\ell_d}$ is achieved at $(k_1,\ldots,k_d)=(m\ell_1,\ldots,m\ell_d)$ (respectively at $(m(\ell_1+1)-1,\ldots,m(\ell_d+1)-1)$), and shifting the indices by one block changes $r_{k_1,\ldots,k_d}$ by at least $m(a_1+\cdots+a_d)$. We will work with $H_{\ell_1,\ldots,\ell_d}$ for $(\ell_1,\ldots,\ell_d) \in \mathcal{L}_4$ for gap principles.

We estimate the number of indices near the boundaries by  
\[\lvert \mathcal{L}_2 - \mathcal{L}_3 \rvert \leq \left(\frac{1}{M_1} + \cdots + \frac{1}{M_d}\right)M_1 \cdots M_d\]
and 
\[\lvert \mathcal{L}_3 - \mathcal{L}_4 \rvert \leq \left(\frac{1}{M_1} + \cdots + \frac{1}{M_d}\right)M_1 \cdots M_d.\]

Therefore, 
\begin{equation}\label{Sec7_eq12}
    \lvert H \rvert \leq \sum_{(\ell_1,\ldots,\ell_d) \in \mathcal{L}_4} \lvert H_{\ell_1,\ldots,\ell_d} \rvert + 3m^d\left(\frac{1}{M_1} + \cdots + \frac{1}{M_d}\right)M_1 \cdots M_d.
\end{equation}
For the right side, 
\begin{equation}\label{Sec7_eq13}
    3m^d\left(\frac{1}{M_1} + \cdots + \frac{1}{M_d}\right)M_1 \cdots M_d \leq 3\left(\frac{1}{M_1} + \cdots + \frac{1}{M_d}\right)N_1 \cdots N_d \leq \frac{\rho}{4}\lvert G \rvert,
\end{equation}
where in the last inequality, we used $M_1,\ldots,M_d \geq J(\rho,d)$ and \eqref{Sec7_eq8}. Combining \eqref{Sec7_eq10}, \eqref{Sec7_eq12}, and \eqref{Sec7_eq13} gives 
\[\sum_{(\ell_1,\ldots,\ell_d) \in \mathcal{L}_4} \lvert H_{\ell_1,\ldots,\ell_d} \rvert \geq \frac{3\rho}{4}\lvert G \rvert.\]

By the pigeonhole principle, there exists some $(\ell_1,\ldots,\ell_d) \in \mathcal{L}_4$ such that 
\begin{equation}\label{Sec7_eq14}
    \lvert H_{\ell_1,\ldots,\ell_d} \rvert \geq \frac{3\rho}{4}\lvert G_{\ell_1,\ldots,\ell_d} \rvert = \frac{3\rho}{4}m^d \geq \frac{\rho}{4}m^d.
\end{equation}

Let 
\[\{P \in \mathcal{P}\:|\:x(P) \in H_{\ell_1,\ldots,\ell_d}\} = \mathcal{S}_H \cup \mathcal{R}_H\]
where 
\[\mathcal{S}_H := \{P \in \mathcal{P}\:|\:x(P) \in H_{\ell_1,\ldots,\ell_d},\:\hat{h}(P) \leq 10\gamma^{-1}\log X\}\]
and 
\[\mathcal{R}_H := \{P \in \mathcal{P}\:|\:x(P) \in H_{\ell_1,\ldots,\ell_d},\:\hat{h}(P) > 10\gamma^{-1}\log X\}.\] 

First, for points in $\mathcal{S}_H$, \eqref{Sec7_eq6} gives 
\begin{equation}\label{Sec7_eq28}
    \lvert \mathcal{S}_H \rvert \leq B(E,d,\rho)^r.
\end{equation}

Let 
\[\mathcal{R}_H = \{P_1,\ldots,P_n\}.\]
We will apply the gap principle Theorem~\ref{gap_principle1_small_s} for these rational points.

Suppose $P_i,P_j \in \mathcal{R}_H$ and let 
\[x(P_i) = \frac{y_i}{s},\quad x(P_j) = \frac{y_j}{s}.\] 
By Lemma~\ref{Weil_canonical_height}, 
\[h(P_i) \geq \hat{h}(P_i)-\log X \geq 9\gamma^{-1}\log X.\]
Therefore, 
\[9h(s) \leq 9\gamma^{-1}\log X \leq h(P_i) \leq h(y_i)\]
It follows that 
\begin{equation}\label{Sec7_eq15}
    \frac{8}{9}h(y_i) \leq h(y_i)-h(s) \leq h(P_i) \leq h(y_i)
\end{equation}
and similarly, 
\begin{equation}\label{Sec7_eq16}
    \frac{8}{9}h(y_j) \leq h(P_j) \leq h(y_j).
\end{equation}
Since $x(P_i),x(P_j) \in H_{\ell_1,\ldots,\ell_d}$, 
\[x(P_i),x(P_j) \geq m(a_1 + \cdots + a_d)\]
and 
\[\lvert x(P_i)-x(P_j) \rvert < m(a_1 + \cdots + a_d).\]
This implies 
\begin{equation}\label{Sec7_eq17}
    y_i,y_j \geq m(a_1 + \cdots + a_d)s
\end{equation}
and 
\begin{equation}\label{Sec7_eq18}
    \lvert y_i-y_j \rvert < m(a_1 + \cdots + a_d)s.
\end{equation}
Without loss of generality, assume $y_i \leq y_j$. Then \eqref{Sec7_eq17} and \eqref{Sec7_eq18} imply 
\[y_i \leq y_j \leq 2y_i.\]
Thus 
\[h(y_i) \leq h(y_j) \leq h(y_i) + \log 2.\]
Recall that we assumed $x(P_i) \geq X^{1/6} \geq 2$. Then $y_i \geq 2s \geq 2$. Thus 
\[h(y_i) \leq h(y_j) \leq 2h(y_i).\]
It follows that 
\[\max\left\{\frac{h(y_i)}{h(y_j)},\frac{h(y_j)}{h(y_i)}\right\} \leq 2.\]
From \eqref{Sec7_eq15} and \eqref{Sec7_eq16}, 
\[\max\left\{\frac{h(P_i)}{h(P_j)},\frac{h(P_j)}{h(P_i)}\right\} \leq \frac{9}{4}.\]
Since $\hat{h}(P_i) > 10\gamma^{-1}\log X$, Lemma~\ref{Weil_canonical_height} implies 
\begin{equation}\label{Sec7_eq19}
    (1-\gamma/10)\hat{h}(P_i) < \hat{h}(P_i) - \log X \leq h(P_i) \leq \hat{h}(P_i) + \log X < (1+ \gamma/10)\hat{h}(P_i)
\end{equation}
and similarly, 
\begin{equation}\label{Sec7_eq20}
    (1-\gamma/10)\hat{h}(P_j) < h(P_j) < (1+\gamma/10)\hat{h}(P_j).
\end{equation}
Therefore, \eqref{Sec7_eq19} and \eqref{Sec7_eq20} imply 
\[\max\left\{\frac{\hat{h}(P_i)}{\hat{h}(P_j)},\frac{\hat{h}(P_j)}{\hat{h}(P_i)}\right\} \leq \frac{1+\gamma/10}{1-\gamma/10}\frac{9}{4}.\]

Now applying Theorem~\ref{gap_principle1_small_s} with $\alpha = \frac{1+\gamma/10}{1-\gamma/10}\frac{9}{4}$, $\delta=1$, and $M = 10\gamma^{-1}$, we have 
\begin{align*}
    \cos \theta_{P_i,P_j} &\leq \frac{1}{2}\sqrt{\frac{1+\gamma/10}{1-\gamma/10}\frac{9}{4}} + \frac{3}{20\gamma^{-1}\gamma} + \frac{4}{10\gamma^{-1}}  \\
    &= \frac{1}{2}\sqrt{\frac{1+\gamma/10}{1-\gamma/10}\frac{9}{4}} + \frac{3}{20} + 0.4\gamma \leq \cos \theta_2.
\end{align*}

Since the angles satisfy $\cos \theta_{P_i,P_j} \leq \cos \theta_2$ for $i \neq j$, the spherical code bound with the choice \eqref{Sec7_eq4} give 
\begin{equation}\label{Sec7_eq29}
    \lvert \mathcal{R}_H \rvert = n \leq B(E,d,\rho)^r.
\end{equation}

Combining \eqref{Sec7_eq28} and \eqref{Sec7_eq29} gives 
\begin{equation}\label{Sec7_eq21}
    \lvert H_{\ell_1,\ldots,\ell_d} \rvert \leq \lvert \mathcal{S}_H \rvert + \lvert \mathcal{R}_H \rvert \leq 2B(E,d,\rho)^r < \frac{\rho}{4}m.
\end{equation}

Combining \eqref{Sec7_eq14} and \eqref{Sec7_eq21} gives 
\[\frac{\rho}{4}m^d \leq \lvert H_{\ell_1,\ldots,\ell_d} \rvert < \frac{\rho}{4}m \leq \frac{\rho}{4}m^d,\]
which is a contradiction.

\subsection{When $h(s) > \gamma^{-1}\log X$}

By \eqref{X_canonical_lower_bound} and \eqref{Sec7_eq1}, we have 
\[h(s) > 17\gamma^{-1} \geq \log L(\delta,\rho,d).\]
So we have $s \geq L(\delta,\rho,d)$. Also we assumed $N_1,\ldots,N_d \geq K(\delta,\rho,d)$.

Let 
\[K := \{r_{k_1,\ldots,k_d} \in x(\mathcal{P})\:|\:g_{k_1,\ldots,k_d} \leq s^\delta\}.\]
By \eqref{Sec7_eq9}, 
\begin{equation}\label{Sec7_eq30}
    \lvert K \rvert \geq \frac{\rho}{2}\lvert G \rvert.
\end{equation}

For each $0 \leq \ell_1 < M_1,\ldots,0 \leq \ell_d < M_d$, let 
\[G_{\ell_1,\ldots,\ell_d} := \{r_{k_1,\ldots,k_d}\:|\:m\ell_1 \leq k_1 < m(\ell_1+1),\ldots,m\ell_d \leq k_d < m(\ell_d+1)\}\]
and let 
\[K_{\ell_1,\ldots,\ell_d} := K \cap G_{\ell_1,\ldots,\ell_d}.\]
Then 
\[\lvert K \rvert \leq \sum_{0 \leq \ell_1 < M_1,\ldots,0 \leq \ell_d < M_d} \lvert K_{\ell_1,\ldots,\ell_d} \rvert + m^d\left(\frac{1}{M_1} + \cdots + \frac{1}{M_d}\right)M_1 \cdots M_d\]

Let 
\[\mathcal{L} := \{0,1,\ldots,M_1-1\} \times \cdots \times \{0,1,\ldots,M_d-1\}\]
be the index set, and define $\mathcal{L}_1,\ldots,\mathcal{L}_4$ as in subsection~\ref{largex_smalls}. The same argument gives 
\begin{equation}\label{Sec7_eq31}
    \lvert K \rvert \leq \sum_{(\ell_1,\ldots,\ell_d) \in \mathcal{L}_4} \lvert K_{\ell_1,\ldots,\ell_d} \rvert + 3m^d\left(\frac{1}{M_1} + \cdots + \frac{1}{M_d}\right)M_1 \cdots M_d.
\end{equation}
Therefore, combining \eqref{Sec7_eq30}, \eqref{Sec7_eq31}, and \eqref{Sec7_eq13} gives
\[\sum_{(\ell_1,\ldots,\ell_d) \in \mathcal{L}_4} \lvert K_{\ell_1,\ldots,\ell_d} \rvert \geq \frac{\rho}{4}\lvert G \rvert.\]

By the pigeonhole principle, there exists some $(\ell_1,\ldots,\ell_d) \in \mathcal{L}_4$ such that 
\begin{equation}\label{Sec7_eq22}
    \lvert K_{\ell_1,\ldots,\ell_d} \rvert \geq \frac{\rho}{4}\lvert G_{\ell_1,\ldots,\ell_d} \rvert = \frac{\rho}{4}m^d.
\end{equation}

Let 
\[\{P \in \mathcal{P}\:|\:x(P) \in K_{\ell_1,\ldots,\ell_d}\} = \mathcal{S}_K \cup \mathcal{R}_K\]
where 
\[\mathcal{S}_K := \{P \in \mathcal{P}\:|\:x(P) \in K_{\ell_1,\ldots,\ell_d},\:\hat{h}(P) \leq 10\gamma^{-1}\log X\}\]
and 
\[\mathcal{R}_K:= \{P \in \mathcal{P}\:|\:x(P) \in K_{\ell_1,\ldots,\ell_d},\:\hat{h}(P) > 10\gamma^{-1}\log X\}.\] 

First, for points in $\mathcal{S}_K$, \eqref{Sec7_eq6} gives 
\begin{equation}\label{Sec7_eq32}
    \lvert \mathcal{S}_K \rvert \leq B(E,d,\rho)^r.
\end{equation}

Let 
\[\mathcal{R}_K = \{P_1,\ldots,P_n\}.\]
We will apply the gap principle Theorem~\ref{gap_principle1_large_s} for these rational points.

Suppose $P_i,P_j \in \mathcal{R}_K$ and let 
\[x(P_i) = \frac{y_i}{s},\quad x(P_j) = \frac{y_j}{s}.\] 
Since $\gcd(y_i,s) \leq s^\delta$, 
\begin{equation}\label{Sec7_eq23}
    (1-\delta)h(y_i) \leq h(y_i)-\delta h(s) \leq h(P_i) \leq h(y_i)
\end{equation}
and similarly, 
\begin{equation}\label{Sec7_eq24}
    (1-\delta)h(y_j) \leq h(P_j) \leq h(y_j).
\end{equation}
The same argument as in subsection~\ref{largex_smalls} gives 
\[\max\left\{\frac{h(y_i)}{h(y_j)},\frac{h(y_j)}{h(y_i)}\right\} \leq 2.\]
From \eqref{Sec7_eq23} and \eqref{Sec7_eq24}, 
\[\max\left\{\frac{h(P_i)}{h(P_j)},\frac{h(P_j)}{h(P_i)}\right\} \leq \frac{2}{1-\delta}.\]
Since $\hat{h}(P_i) > 10\gamma^{-1}\log X$, Lemma~\ref{Weil_canonical_height} implies 
\begin{equation}\label{Sec7_eq25}
    (1-\gamma/10)\hat{h}(P_i) < \hat{h}(P_i) - \log X \leq h(P_i) \leq \hat{h}(P_i) + \log X < (1+ \gamma/10)\hat{h}(P_i)
\end{equation}
and similarly, 
\begin{equation}\label{Sec7_eq26}
    (1-\gamma/10)\hat{h}(P_j) < h(P_j) < (1+\gamma/10)\hat{h}(P_j).
\end{equation}
Therefore, \eqref{Sec7_eq25} and \eqref{Sec7_eq26} imply 
\[\max\left\{\frac{\hat{h}(P_i)}{\hat{h}(P_j)},\frac{\hat{h}(P_j)}{\hat{h}(P_i)}\right\} \leq \frac{1+\gamma/10}{1-\gamma/10}\frac{2}{1-\delta}.\]

Now applying Theorem~\ref{gap_principle1_large_s} with $\alpha = \frac{1+\gamma/10}{1-\gamma/10}\frac{2}{1-\delta}$ and $M = 10\gamma^{-1}$, we have 
\begin{align*}
    \cos \theta_{P_i,P_j} &\leq \frac{1}{2}\sqrt{\frac{1+\gamma/10}{1-\gamma/10}\frac{2}{1-\delta}} + \frac{3\delta}{2(1-\delta-\gamma)} + \frac{4}{10\gamma^{-1}} \\
    &= \frac{1}{2}\sqrt{\frac{1+\gamma/10}{1-\gamma/10}\frac{2}{1-\delta}} + \frac{3\delta}{2(1-\delta-\gamma)} + 0.4\gamma \leq \cos \theta_3.
\end{align*}

Since the angles satisfy $\cos \theta_{P_i,P_j} \leq \cos \theta_3$ for $i \neq j$, the spherical code bound with the choice \eqref{Sec7_eq5} give 
\begin{equation}\label{Sec7_eq33}
    \lvert \mathcal{R}_K \rvert = n \leq B(E,d,\rho)^r.
\end{equation}

Combining \eqref{Sec7_eq32} and \eqref{Sec7_eq33} gives 
\begin{equation}\label{Sec7_eq27}
    \lvert K_{\ell_1,\ldots,\ell_d} \rvert \leq \lvert \mathcal{S}_K \rvert + \lvert \mathcal{R}_K \rvert \leq 2B(E,d,\rho)^r < \frac{\rho}{4}m.
\end{equation}

Combining \eqref{Sec7_eq22} and \eqref{Sec7_eq27} gives 
\[\frac{\rho}{4}m^d \leq \lvert K_{\ell_1,\ldots,\ell_d} \rvert < \frac{\rho}{4}m \leq \frac{\rho}{4}m^d,\]
which is a contradiction.

\section{Applications}\label{Applications}

In this section we present several consequences of our main theorem. Informally, the theorem shows that the $x$-coordinates of rational points on an elliptic curve cannot exhibit strong additive structure. In particular, large sets of rational points cannot have their $x$-coordinates concentrated inside low-dimensional additive configurations such as generalized arithmetic progressions.

We begin with a zero-density formulation of Theorem~\ref{main_theorem}. We then derive consequences under a small sumset condition or a large additive energy condition, using Freiman's theorem or Balog-Szemer\'{e}di-Gowers theorem.

In all the results below, the constants are effectively computable and the dependence of the constants on $E$ disappears if one assumes Lang's conjecture (Conjecture~\ref{Lang_conjecture}).

\subsection{Zero density inside generalized arithmetic progressions}

We first reformulate Theorem~\ref{main_theorem} as a zero-density statement inside generalized arithmetic progressions. The main theorem shows that a proper $d$-dimensional generalized arithmetic progression cannot contain a positive proportion of $x$-coordinates of rational points of $E/\bbq$ once its size is sufficiently large. 

\begin{corollary}\label{zero_density}
    Let $E/\bbq$ be an elliptic curve of Mordell-Weil rank $r \geq 1$. Let $G_n$ be a sequence of proper $d$-dimensional generalized arithmetic progressions in $\bbq$ such that $\lvert G_n \rvert \rightarrow \infty$. Then
    \[\frac{\lvert x(E(\bbq)) \cap G_n \rvert}{\lvert G_n \rvert} \rightarrow 0.\]
\end{corollary}

\subsection{Small sumsets}

We now consider situations where the set of $x$-coordinates has small additive growth. In additive combinatorics it is known that sets with small doubling must possess strong algebraic structure. More precisely, Freiman's theorem implies that such sets are contained in generalized arithmetic progressions of bounded dimension and positive density.

Combining this structural result with Theorem~\ref{main_theorem} yields the following corollary.

\begin{corollary}\label{Freiman_cor}
    Let $E/\bbq$ be an elliptic curve of Mordell-Weil rank $r \geq 1$ and $\mathcal P\subset E(\bbq)$ a finite subset. Put $S=x(\mathcal P)$.

    Suppose that
    \[\lvert S+S \rvert \le K\lvert S \rvert\]
    for some constant $K$.

    Then there exists a constant $A(E,K)>0$ such that
    \[\lvert \mathcal P \rvert \le A(E,K)^r.\]
\end{corollary}
\begin{proof}
    By Freiman's theorem (\cite{Fre73}, \cite{GR07}, \cite[Theorem~5.33]{TV10}), the set $S$ is contained in a proper generalized arithmetic progression $G$ of dimension $d(K)$ and satisfies 
    \[\lvert S \rvert \geq \rho(K)\lvert G \rvert.\]
    Applying Theorem~\ref{main_theorem} completes the proof.
\end{proof}

\subsection{Large additive energy}

From a Diophantine perspective, it is natural to ask whether the $x$-coordinates of rational points on an elliptic curve can satisfy many additive relations.

More precisely, given a finite set of rational points $\mathcal P \subset E(\bbq)$, one may consider the number of solutions to the equation

\[x(P_1)+x(P_2)=x(P_3)+x(P_4), \quad P_i\in\mathcal P.\]

If many such relations occur, the set $x(\mathcal P)$ exhibits strong additive correlations. In additive combinatorics this phenomenon is quantified by the additive energy of the set.

For a finite set $S \subset \bbq$, define the additive energy by 
\[E(S)= \lvert \{(a,b,c,d) \in S^4\:|\:a+b=c+d\} \rvert.\]

The following result shows that large collections of rational points on an elliptic curve cannot exhibit large additive energy.

\begin{corollary}\label{Energy_cor}
    Let $E/\bbq$ be an elliptic curve of Mordell-Weil rank $r \geq 1$ and let $\mathcal P\subset E(\bbq)$ be a finite subset. Put $S=x(\mathcal P)$.

    Suppose that
    \[E(S) \ge \frac{\lvert S \rvert^3}{K}\]   
    for some constant $K$.

    Then there exists a constant $A(E,K)>0$ such that
    \[\lvert \mathcal P \rvert \leq A(E,K)^r.\]
\end{corollary}
\begin{proof}
    By the Balog-Szemer\'{e}di-Gowers theorem (\cite{BS94}, \cite{Gow98}, \cite[Theorem~2.31]{TV10}), there exists a subset $S'\subset S$ with $|S'|\gg_K |S|$ and
    \[\lvert S'+S' \rvert \ll_K \lvert S'\rvert.\]

    Applying Corollary~\ref{Freiman_cor} to the set $S'$ and using the bound $|S'|\gg_K |S|$ proves the Corollary.
\end{proof}

\end{document}